\documentclass[11pt,a4paper]{article} 
\usepackage{amsmath,amsfonts,latexsym,amsthm,amssymb}
\usepackage{graphicx}
\usepackage{color}
\usepackage[text={15cm,24cm},top=3cm]{geometry}  
\usepackage{hyperref}

\newcommand{\N}{{\mathbb{N}}}  
\newcommand{\R}{{\mathbb{R}}}  
\newcommand{\C}{{\mathbb{C}}}  
\newcommand{\id}{{\mathrm{d}}}  
\newcommand{\Hy}{{\mathbb{H}}} 

\newcommand{\funcg}{g}
\newcommand{\metricg}{\mathsf{g}}

\DeclareMathOperator{\arccosh}{arccosh}
\DeclareMathOperator{\arcsinh}{arcsinh}
\DeclareMathOperator{\sign}{sign}
\DeclareMathOperator{\2F1}{{}_2F_1}

\theoremstyle{plain}
\newtheorem{thm}{Theorem}[section]
\newtheorem{lem}[thm]{Lemma}

\newtheorem{cor}[thm]{Corollary}

\theoremstyle{remark}

\newtheorem{exm}[thm]{Example}

\theoremstyle{definition}
\newtheorem{dfn}[thm]{Definition}
\newtheorem{notation}[thm]{Notation}
\author{Kamil Mroz, Alexander Strohmaier}
\title{Explicit bounds on eigenfunctions and spectral functions on manifolds hyperbolic near a point}

\begin{document}
\maketitle
\begin{abstract}
 We derive explicit bounds for the remainder term in the local Weyl law for locally hyperbolic
 manifolds, we also give the estimates of the derivative of this remainder. We use these to obtain explicit bounds for the
 $C^l$-norms of the $L^2$-normalised eigenfunctions in the case spectrum of the Laplacian is discrete, e.g. 
for closed Riemannian manifolds. 
 We also derive bounds for the local heat trace. Our estimates are purely
 local and therefore also hold for any manifold at points near which the metric is locally hyperbolic.
\end{abstract}

\section{Introduction}

Let $(M,\metricg)$ be an n-dimensional Riemannian manifold and as usual denote by $C^\infty_0(M)$ the 
space of compactly supported smooth functions and by $L^2(M)$ the space of square integrable functions, i.e.
the completion of $C^\infty_0(M)$ in the $L^2$-inner product $\langle f_1,f_2 \rangle=\int_M f_1(x) \overline{f_2(x)} \, \id \mathrm{vol}_\metricg(x)$,
where $\id \mathrm{vol}_\metricg$ is the Riemannian volume form. Suppose that $\Delta$ is a non-negative (in the sense of operator theory) 
self-adjoint extension of the Laplace operator acting on $C^\infty_0(M)$. Note that such self-adjoint extensions always exist and are 
unique in case the manifold is complete. We do not however assume completeness here.

By the spectral theorem for unbounded self-adjoint operators, there exists a spectral family $E_\lambda$
with associated spectral measure such that
$$
 \Delta = \int_\mathbb{R} \lambda^2 \, \id E_\lambda.
$$
For each $\lambda \in \mathbb{R}$ the operator $E_\lambda$ is a projection and since
$\Delta^N E_\lambda$ is $L^2$-bounded for any $N>0$ it follows from elliptic regularity that $E_\lambda$
has a smooth integral kernel $e_\lambda \in C^\infty(M \times M)$, i.e.
$$
 (E_\lambda f)(x) = \int_M e_\lambda(x,y) f(y) \, \id \mathrm{vol}_g(y).
$$
The restriction of $e_\lambda$ to the diagonal in $M \times M$ is called the local counting function, and is denoted by
\begin{equation} \label{def1}
 N_x(\lambda) := e_\lambda(x,x).
\end{equation}
Note that in the case that $\Delta$ has compact resolvent the spectrum is purely discrete
and there exists an orthonormal basis $(\varphi_j)_{j \in \mathbb{N}_0}$ in $L^2(M)$ consisting of 
eigenfunctions with eigenvalues $\lambda_j^2$ for $\Delta$:
$$
 \Delta \varphi_j = \lambda_j^2 \varphi_j. 
$$
The local counting function is then given by
$$
 N_x (\lambda) = \sum_{\lambda_j \leq \lambda} | \varphi_j(x) |^2,
$$
and integration over $M$ gives the usual counting function
$$
 N(\lambda) = \int_M N_x(\lambda) \, \id \mathrm{vol}_\metricg(x) = \# \{ \lambda_j \leq \lambda\}.
$$
The local counting function is well defined independently of any assumptions on the discreteness of the spectrum.

The famous local Weyl law states that
$$
 N_x(\lambda) = \frac{\omega_n}{(2 \pi)^n} \lambda^{n} + O(\lambda^{n-1}),
$$
where $\omega_n$ is the volume of the unit ball in $\mathbb{R}^n$. The proof of this result in the case of closed Riemannian manifolds is due to Levitan. In 1968 H\"ormander~\cite{Hormander2} generalised it to the case of pseudodifferential operators of order $m$.  
Various improvements on the remainder term are known in case of compact manifolds with negative sectional curvature or
under assumptions on the nature of the dynamics of the geodesic flow. In particular, in 1975 Duistermaat and Guillemin~\cite{Duistermaat} proved that the estimate $O(\lambda^{n-1})$ cannot be improved when the geodesic flow is periodic. They also showed that for boundary-less manifolds the remainder is $o(\lambda^{n-1})$, when the set of periodic geodesics has Liouville measure 0. Two years later B\'erard~\cite{Berard} obtained a logarithmic improvement for manifolds with nonpositive sectional curvature (for surfaces it suffices to have no conjugate points).  

Note that the original local Weyl law immediately implies a uniform bound on eigenfunctions:
$$
 \| \varphi_j \|_{L^\infty(M)} \leq C_M \lambda_j^{\frac{n-1}{2}},
$$
for some $C_M>0$ in the case that $\Delta$ has compact resolvent. This estimate is sharp without additional assumptions,
as the example of the sphere shows. Further Sogge~\cite{Sogge} showed that without any additional assumptions on $M$ we have the following bound on $L^p$ norms:
\begin{eqnarray*}
& \| \varphi_j \|_{L^p(M)} \leq C_{M,L^p} \lambda_j^{\delta(n,p)},&\\
& \delta(n,p)= 
 \begin{cases}
 \frac{n-1}{2}- \frac{n}{p}, &  \frac{2(n+1)}{n-1} \leq p \leq \infty, \\
 \frac{n-1}{4}- \frac{n-1}{2p}, &2 \leq p \leq \frac{2(n+1)}{n-1} .
\end{cases} &
\end{eqnarray*}

It is natural that B\'erard's logarithmic improvement of the local counting function leads to the logarithmic improvements of the $L^p$ estimates. In~\cite{Tacy} Hassell and Tacy proved that under the same assumptions as in B\'erard~\cite{Berard}, one has:
$$
 \| \varphi_j \|_{L^p(M)} \leq C_{M,L^p} \frac{\lambda_j^{\frac{n-1}{2}-\frac{n}{p}}}{(\log \lambda_j)^{\frac{1}{2}- \frac{n+1}{p(n-1)}}}.
$$
Similarly, one can derive a bound on the $C^l$-norm
$$
 \| \varphi_j \|_{C^l(M)} \leq C_{l,M} \lambda_j^{l+\frac{n-1}{2}}.
$$
Whereas such  estimates of eigenfunctions are very useful for
 proofs and general considerations for purposes of numerical
analysis one often needs an explicit value of the constant.
The aim of this article is to provide explicit estimates of the 
eigenfunctions and local counting function of the Laplacian for manifolds
that are hyperbolic near the point in question.
More specifically, suppose that $\Delta$ has compact resolvent, $x \in M$ and the open metric ball $B_\rho(M,x)$ of radius $\rho>0$
centred at $x$ is locally isometric to 
hyperbolic $n$-space. We find constants $C_{1,M}$ depending only on $\rho$ such that
$$
| \nabla \varphi_j(x) | \leq C_{1,M} \lambda^{\frac{n+1}{2}}.
$$
We also present an algorithm for computing $C_{l,M}$ in the case of surfaces. As an example we give formulas for $C_{l,M}$ $(l =1,2, \ldots, 8)$ in this context (for bigger $l$ formulas get very complicated).  Our estimates are derived from estimates for the local counting function.

The paper is organised as follows: first, we present the classical formula for the solution to the shifted wave equation~(\ref{wave}) in hyperbolic space and relate it to the integral kernel of $\cos(\sqrt{\Delta -(n-1)^2/4} \, t)$; through this paper we call this the shifted wave kernel. Second, through the analysis of the mentioned  wave kernel we present the generalisation of the Mehler-Fock formula (Theorem~\ref{transform}), which gives us a way to compute the value of the integral kernel of $\int\cos(\sqrt{\Delta -(n-1)^2/4}\, t) \funcg(t) \id t$ on the diagonal, where $\funcg$ is a test function. Next, by finite propagation speed and the Fourier-Tauberian argument given in~\cite{Safarov} we are able to show our main result: estimates on the local counting function of the Laplace operator (Theorem~\ref{main}). As consequences of this result we show bounds on the local heat trace and estimate the growth of eigenfunctions of the Laplacian as well as growth of their first derivative. At the end we show how to find estimates of higher derivatives of the eigenfunctions for the hyperbolic surfaces.

\section[The wave kernel]{The integral kernel of the solution operator of the shifted wave equation}
In order to study the local counting function of the Laplace operator at a point $x \in M^n$ near which $M^n$ is hyperbolic, we first need to study the shifted wave kernel on $\Hy^n$. Here is a motivation: Safarov in~\cite{Safarov} (Theorem 1.3) showed that for the suitably fast decreasing functions $\varrho_1$, $\varrho_2$ one has:
$$
|N_x(\lambda) - N_x \ast \varrho_1(\lambda)| \leq N'_x \ast \varrho_2(\lambda)
$$
The Weyl law implies that the local counting function has polynomial growth, therefore the right hand side of the above inequality is of lower order, therefore we may treat is as correction term in the expansion. In addition if the Fourier transforms of functions $\varrho_1$, $\varrho_2$ is compactly supported, then $N_x \ast \varrho_1$ and $N'_x \ast \varrho_2$ depend only on values of the Fourier transforms  of $N_x$ and $N_x'$ on the  on the support of $\varrho_1$ and $\varrho_2$ respectively. Moreover by a simple modification of the counting function, one does not need the full information about the Fourier transform of $N_x$ and $N'_x$ but it is enough to know just the cosine transform  of $N'_x$. By the lemma below we will show how it is connected to the wave kernel.
\begin{lem}\label{lem1}
The cosine transform of the derivative of the local counting function of the Laplacian on the manifold $M$ is the diagonal of the wave kernel on $M$.
\end{lem}
\begin{proof}
By the spectral theorem, for any $\lambda$, the projection $E_\lambda$ may be expressed in terms of its smooth integral kernel. By the definition of functional calculus for any $L^2$ function $f$ we have
$$
\langle f , \cos(\sqrt{\Delta}t) f \rangle = 
\int_\R \cos(\lambda t)\, \id \int_{M \times M} e_\lambda(x,y) f(y) \overline{f(x)} \,\id \mathrm{Vol_\funcg}(y)\id \mathrm{Vol_\funcg}(x).
$$
Of course we need to understand this equality in the distributional sense since $E_\lambda$ is a distribution in $\lambda$. When we consider the diagonal of $e_\lambda$ the proof is complete.
\end{proof}
Let us move to the description of the wave kernel.
As a model of hyperbolic space we will use the upper half space, i.e.
\[ \Hy^n = \{ x \in \R^n : x_n >0 \},\]
with the metric $\metricg = x_n^{-2}(\id x_1^2 + \id x_2^2+ \ldots + \id x_n^2)$. The Laplacian in this coordinate system is given by
\begin{eqnarray*}
\Delta & = & - x_n^2 ( \partial_{x_1}^2+ \ldots+  \partial_{x_n}^2)+ (n-2)x_n  \partial_{x_n},
\end{eqnarray*}
where by $\partial_{x_i}$ we mean $\partial/\partial_{x_i}$. Introduce a function $u \colon \Hy^n \times \Hy^n \to \R_+$ defined by
\begin{equation} \label{u}
 u(x,y)= \frac{\| x-y\|^2}{4 x_n y_n},
 \end{equation}
where $\| \, \cdot \, \|$ denotes the Euclidean norm.  The hyperbolic distance $\rho: \Hy^n \times \Hy^n \to \R_+$ is given by  the relation
\begin{equation} \label{rho}
1+2 u= \cosh(\rho).
\end{equation}
The function $u$ is commonly used as a replacement of the distance function as some of the formulae become simpler.

Suppose now that $f$ is a function on $\Hy^n$, $x \in \Hy^n$. We want to produce the rotational symmetrization of $f$ about $x$. Define
\[ 
f_x^{\circ}(\rho):= \frac{1}{A(\rho)} \int_{S_{\rho}(x)} f(s) \, \id \omega(s), \quad \rho > 0,
\] 
where  $S_{\rho}(x)$ is the geodesic sphere in $\Hy^n$ with the centre at $x$ and the radius $\rho$, $A(\rho)$ is its area, $\id \omega$ is the area element of  $S_{\rho}(x)$, wherever this expression makes sense. By definition, $f_x^{\circ}(0) =f(x)$. We call $f_x^{\circ}$ the radialization of $f$ about $x$.

Denote by $L_{\Hy^n}$ the shifted Laplace operator $\Delta_{\Hy^n} -\left( \frac{n-1}{2} \right)^2$. Consider the Cauchy problem for the shifted wave equation in hyperbolic $n$-space:
\begin{equation} \label{wave}
\left\{
\begin{array}{r c l}
-L_{\Hy^n} v &=& \frac{\partial^2 v}{\partial t^2}  ,\\
\left. v \right|_{t=0} &=& f, \\
 \left. \frac{\partial v}{\partial t} \right|_{t=0} &=& 0,
\end{array} \right.
\end{equation}
for a given $f \in L^2(\Hy^n)$. This problem may be solved using functional calculus: 
\begin{equation} \label{solution}
v(x,t) = \int_\R \cos(\sqrt{\lambda^2-(n-1)^2/4})\, \id E_\lambda f(x)=:\cos\left( \sqrt{L_{\Hy^n}} t \right) f(x).
\end{equation}
Because the family of operators $\cos( \sqrt{L_{\Hy^n}}t)$ is bounded on $L^2(\Hy^n)$ for all $t \in \R$, the Schwartz kernel theorem implies that it has an associated distributional kernel $\tilde{k}(x,y,t)$ which we call the shifted wave kernel.
On the other hand, there are well known formulae for the solution to (\ref{wave}) (see e.g. ~\cite{Grigor}):
\begin{eqnarray}
 \label{oddsol} v(x,t) &\!\!\!\!\!\!= & \!\!\!\!\!\! \frac{\partial_t}{(2m-1)!! } \left( \left( \frac{1}{\sinh t}\partial_t \right)^{m-1} \left( \sinh^{2m-1}(t) f_x^{\circ}(t) \right) \right)\!\!, \  n= 2m+1, \\
\label{evensol} v(x,t) &\!\!\!\!\!=&\!\!\!\!\!\! \frac{\partial_t}{2^{m+\frac{1}{2}} m!  }\! \int_0^t \!\! \frac{\sinh \rho}{\sqrt{\cosh t- \cosh \rho}} \! \left(\! \frac {1}{\sinh \rho} \partial_\rho \! \right)^{m}\!\! \left(  \sinh^{2m}(\rho) f_x^{\circ}(\rho) \right) \id \rho, \ n= 2m+2.
\end{eqnarray}
Since we do not need the exact formula for the shifted wave kernel but just local (in time) data, let us pair the solution to the wave equation with a test function $\funcg \in C^{\infty}_0(\R)$: 
\begin{eqnarray*}
\int_\R\!\! v(x,t)  \funcg( t) \, \id t &\!\!\!\! =&  \!\!\!\!
\begin{cases}
\frac{2(-1)^m\omega_{2m+1}^{-1}}{(2m+1)!! } \int_{\Hy^n}  f(y) \left( \frac{1}{\sinh \rho} \partial_\rho \right)^{m} \funcg(\rho) \,\id \mu(y), & \!\!n=2m+1,\\
\frac{(-1)^{m+1} \omega_{2m+2}^{-1}}{2^{m+\frac{1}{2}} (m+1)!  }  \int_{\Hy^n} f(y) \left( \frac{1}{\sinh \rho} \partial_\rho \right)^{m} \! \int_{\rho}^{\infty} \!\! \frac{\funcg'(t) \, \id t }{\sqrt{\cosh t- \cosh \rho}} \, \id \mu(y), & \!\!n=2m+2,
\end{cases}
\end{eqnarray*}
where $\rho$ denotes the hyperbolic distance from $x$ to $y$. Now, from the formula above and (\ref{solution}), one can easily extract the formula for the integral kernel of the operator \\$\int_\R \cos\left( \sqrt{L_{\Hy^n} }t \right) \funcg(t) \, \id t$, denote it by $\tilde{k}_{n,\funcg}$. Recalling that $\omega_n= \pi^{n/2}/\Gamma(n/2+1)$, we may summarise the calculations above by the following lemma:
\begin{lem}
The integral kernel $\tilde{k}_{n,g}$ is given by
\begin{eqnarray} \label{odd}
\tilde{k}_{2m+1,\funcg}(x,y) &= &(-2 \pi)^{-m} \left( \frac{1}{\sinh \rho} \partial_\rho \right)^{m} \funcg(\rho), \\
\label{even}
\tilde{k}_{2m+2,\funcg}(x,y) &=& \frac{\sqrt{2}}{ (-2 \pi)^{m+1}}   \left( \frac{1} {\sinh \rho} \partial_\rho \right)^{m} \!\!\! \int_{\rho}^{\infty} \!\!\!\!\!\!\!\ \frac{\funcg'(t) \, \id t }{\sqrt{\cosh t- \cosh \rho}} 
\end{eqnarray} 
where $\rho = \mathrm{dist(x,y)}$.
\end{lem}
\subsection{The shifted wave kernel  in $\Hy^n$ on the diagonal}
The question we want to answer is: how does the integral kernel $\tilde{k}_{n,\funcg}$ look on the diagonal, where $\rho=0$. As we mentioned in the introduction, the local counting function is smooth, therefore the wave kernel paired with some test function is also smooth.Thus it makes sense to talk about the value of $\tilde{k}_{n,\funcg}$ at a point. Since $\Hy^n$ is a homogeneous space, this value does not depend on $x$. We cannot examine this directly from (\ref{odd}), (\ref{even}), because the operator $ \frac{1}{\sinh \rho} \partial_\rho$ is singular for $\rho =0$. Moreover the integral does not satisfy the assumptions of the fundamental theorem of calculus there.

Instead, we manipulate $\tilde{k}_{n,\funcg}$. The case of $n=2$ it is well known (see e.g. (70) and (114) in~\cite{Marklof}
and~\cite{Iwaniec}):
\[ \tilde{k}_{2,\funcg}(x,x) = \frac{1}{4 \pi} \int_{\R} h(t)  \tanh( \pi t) t \, \id t,\]
where $h$ is the cosine transform of $\funcg \in C^\infty_0(\R)$, i.e.,
\begin{equation} \label{h}
h(t):= \int_\R \funcg(x) \cos( t x) \, \id x.
\end{equation}
It is easier to express $\tilde{k}_{n,g}$ in terms of the cosine transform, $h$, of $\funcg$ rather than in terms of $\funcg$. In addition, it turns out that the change of coordinates $u = \sinh^2(\rho/2)$ makes the calculations much simpler. Now, using (\ref{even}), (\ref{odd}) we obtain recursion formulas in terms of the point pair invariant $u$\begin{equation} \label{dotw}
\begin{cases}
\tilde{k}_{1,\funcg}(x,y)= \funcg(\rho(u)),\\
\tilde{k}_{2,\funcg}(x,y)= -\displaystyle{\frac{1 }{ 2\sqrt{2}}  \! \int_{\rho(u)}^{\infty} \!\!\!\!\ \frac{\funcg'(t) \, \id t }{\sqrt{\cosh t- 1-2u}}}, \\
\tilde{k}_{2m+1,\funcg}(x,y) = \left({-4 \pi} \right)^{-m}  {\partial_ u^m}\tilde{k}_1(x,y), & \textrm{for } m \in \N_+,\\
\tilde{k}_{2m+2,\funcg}(x,y) = \left({-4 \pi} \right)^{-m}  {\partial_ u^m}\tilde{k}_2(x,y), & \textrm{for } m \in \N_+.
\end{cases} 
\end{equation}
The explicit formula for $\tilde{k}_{2,g}$ is well known (see e.g.~\cite{Iwaniec}). It is expressed in terms of the  function $h$, of course, and the main tool to recover $h$ from $\tilde{k}$ for $n=2$ is the Mehler-Fock formula. The following theorem accomplishes this for general $n$. This result may also be proved using Harish-Chandra's Plancherel formula for spherical functions, see~\cite{Chavel}, p. 292.  For the sake of completeness and since we will derive through the proof a generalization of the Mehler-Fock formula, which is interesting in its own right, we give the proof here. Moreover, the transform which we developed allow us to give the estimate of the derivatives of the eigenfunctions.
\begin{lem} \label{eventrace}
Suppose $\funcg \in C_0^{\infty}(\R)$. The integral kernel, $\tilde{k}_{n,\funcg}$, on the diagonal is given by
\begin{equation*}
\tilde{k}_{n,\funcg}(x,x)= \begin{cases}
\displaystyle{\frac{1}{(4 \pi)^{m+1}m!} \int_{\R} h(t)  \tanh( \pi t) t \left( \frac{1}{2} + i t \right)_{\!\!m} \left( \frac{1}{2} - i t \right)_{\!\!m} \id t,} & n=2m+2\\
\displaystyle{\frac{1}{(2 \pi)^{m+1}( 2m-1)!!}  \int_\R   h(t) \left(i t \right)_m\left(-i t \right)_m \, \id t,} & n=2m+1,
\end{cases}
\end{equation*}
where $h$ is the cosine transform of $\funcg$ and $(x)_n$ is the Pochhammer symbol, i.e. $(x)_n= \Gamma(x+n)/ \Gamma(x)=x(x+1)\ldots(x+n-1)$. 
\end{lem}
\begin{proof}
Assume that $\varphi$ is a generalised eigenfunction of the Laplacian and is radial about $x$. The eigenvalue equation, $\Delta \varphi = \lambda^2 \varphi$, in geodesic polar coordinates about a point $x$ is:
\[
\varphi''(\rho)+ (n-1) \coth(\rho) \varphi'(\rho) + \lambda^2 \varphi(\rho)=0.
\]
It is a second order linear differential equation, so its general solution is a linear combination of two independent solutions, $f_1,\ f_2$. Standard theory shows that we may set $f_1(0)=1$, where $f_1$ is an entire function of $\rho$, and $f_2$ is well defined for $\rho >0 $ and has a singularity for $\rho=0$ of the logarithmic type for $n=2$ and of type $\rho^{2-n}$ for $n>2$.  \\
After a substitutions $u= 1/2 (\cosh(\rho)-1)$, $ \lambda^2 = (n-1)^2/4 + t^2$ we obtain
\begin{equation} \label{eigenequation}
u(u+1) \varphi''(u) + n(n +1/2) \varphi'(u) + ((n-1)^2/4 + t^2)\varphi(u) = 0.
\end{equation}
Recall that the hypergeometric function $F(\lambda, \beta, \gamma, z)$ is defined to be the solution of the differential equation
\[ z(1-z) F''(z) - ((\alpha+\beta+1)- \gamma) F'(z)- \alpha \beta F(z) =0, \]
with initial condition $F(0)=1$. Therefore, the function 
\[
F_{t,n}(u):=\2F1((n-1)/2 +it,(n-1)/2 -it ,n/2, -u)
\]
solves $(\ref{eigenequation})$ with $F_{t,n}(0)=1$. This implies that the radialization about $x$ of the generalised eigenfunction, $\varphi$, corresponding to the eigenvalue $\lambda^2= (n-1)^2/4 +t^2$ is given by
\[ 
\varphi_x^{\circ}(\rho)= F_{t,n}(u(x,y))\varphi(x).
\]
Recall that the function $\tilde{k}_n$ depends only on $u$, so we may write 
\begin{equation} 
\int_{\Hy^n} \tilde{k}_n(x,y) \varphi(y)\, \id \mu (y)= \int_{\Hy^n}\tilde{k}_n(x,y) \varphi_x^{\circ}(\rho)\, \id \mu (y)=\varphi(x) \int_{\Hy^n} \tilde{k}_n(u) F_{t,n}(u)\, \id \mu (y),
\end{equation}
where $\tilde{k}_n(u)$, $F_{t,n}(u)$ denote $\tilde{k}_n(x,y)$, $F_{t,n}(u(x,y))$  respectively. Let $x= (0,\ldots, 0,1)$ and introduce geodesic polar coordinates around $x$. The spectral theorem implies that 
\begin{equation}\label{H}
h(t)=\frac{\omega_{n}}{n} 2^{n-1} \int_{0}^{\infty} \tilde{k}_n( u) \2F1 \left(\frac{n-1}{2} + i t, \frac{n-1}{2} - i t , \frac{n}{2}, -u \right) (u^2+u)^{(n-2)/2} \,  \id u,
\end{equation}
where $\omega_n$ denotes the volume of an Euclidean unit ball in $\R^n$ and $h$ is as before the cosine transform of $\funcg \in C_0^{\infty}(\R)$. The inverse transform for $n=2$ is as follows (see e.g.\cite{Iwaniec}):
\begin{equation} \label{inverse}
\tilde{k}_2(x,y)= \frac{1}{4 \pi} \int_{\R} h(t) \2F1\left(\frac{1}{2}+ i t,\frac{1}{2}- i t, 1, -u(x,y)\right) \tanh( \pi t) t \, \id t, \quad \textrm{for } n=2. 
\end{equation}
Note that the hypergeometric function $\2F1(\alpha,\beta,\gamma,z)$ has an analytic continuation over the plane $\C$ cut along $[1,+\infty)$, given by the integral representation:
\[
\2F1(\alpha, \beta, \gamma,z)= \frac{\Gamma(\gamma)}{\Gamma(\beta) \Gamma(\gamma-\beta)} \int_0^1 t^{\beta-1} (1-t)^{\gamma-\beta-1} (1-tz)^{-\alpha} \id t.
\] 
Moreover the $m$-th derivative of $\2F1$ is 
\begin{equation}\label{diff}
\partial_z^m\2F1(\alpha, \beta, \gamma,z)= \frac{\prod_{j=0}^{m-1} (\alpha+j) \prod_{j=0}^{m-1}( \beta+j)}{\prod_{j=0}^{m-1} (\gamma+j)} \2F1(\alpha+im, \beta+m, \gamma+m,z).
\end{equation}
Since the hypergeometric function satisfies the initial condition $F_{t,n}(0)=1$, the recursion formula (\ref{dotw}) and (\ref{inverse}) end the proof for even $n$. For odd $n$ note that 
\begin{equation}\label{cosinus}
\2F1(it, -it ,1/2, -u)= \cos(2 t \arcsinh\sqrt{u}).
\end{equation}
Substitute $\rho= 2 \arccosh \sqrt{u}$ in  the integrals (\ref{H}) and (\ref{odd}) for $m=0$. Then one can easily recognise that these formulas simplify to the cosine transform and its inverse. Now use identities (\ref{H}) and (\ref{cosinus}) to get that 
\begin{equation}
\tilde{k}_1(x,y)= \frac{1}{2 \pi}  \int_\R h(t) \2F1(it, -it ,1/2, -u(x,y)) \, \id t.
\end{equation} 
 The recursion formula (\ref{dotw}) and the differential identity (\ref{diff}) imply
$$
\tilde{k}_{2m+1}(x,y)= \frac{(2 \pi)^{-m-1}}{(2m-1)!!} \int_\R h(t) (it)_m(-it)_m  \2F1(m+it,m-it,m+1/2,-u(x,y))\, \id t.
$$
When we take into account the initial condition $F_{t,n}(0)=1$ the proof is complete.
\end{proof}
The proof of Theorem~\ref{eventrace} gives us class of transformations which contains the cosine  transform as well as the Mehler-Fock transform. Further, we know that the class of functions of interest to us are in the domain of these transforms.
\begin{thm}[Generalised Mehler-Fock formula] \label{transform} 
Suppose  $\funcg \in C_0^{\infty}(\R)$ and $\funcg$ is even. Define $f_e$, $f_o$ by
\[ 
f_e(u)= -\frac{1 }{ 2\sqrt{2}}  \! \int_{\rho(u)}^{\infty} \!\!\!\!\ \frac{\funcg'(t) \, \id t }{\sqrt{\cosh t- 1-2u}}, \quad f_o(u)= \funcg(\rho(u)) 
\]
where $\rho(u)= 2 \arccosh \sqrt{u}$. Then the following inverse formulae hold:
\begin{eqnarray}\label{inversen}
 f_e(u)&\!\!\!\!\!=&\!\!\!\!\! \Gamma(m+1)^{-2} \int_{\R}  \int_{0}^{\infty} f_e\left( v \right)  \2F1 \left(\frac{1}{2} +m+ i t, \frac{1}{2} +m - i t , m+1, -v \right) (v^2+v)^{m} \,  \id v \times \nonumber \\
&&\!\!\!\!\!\!\! \2F1\left( \frac{1}{2}+m+ i t,\frac{1}{2}+ m- i t, m+1, -u \right) \tanh( \pi t) t \left( \frac{1}{2} + i t \right)_{\!\!m} \left( \frac{1}{2} - i t \right)_{\!\!m} \, \id t, \\
 f_o(u)&\!\!\!\!\!=&  \!\!\!\!\! \Gamma\left(m+\frac{1}{2}\right)^{-2}\int_{\R}  \int_{0}^{\infty} f_o \left( v \right)  \2F1 \left(m+ i t, m - i t , m+\frac{1}{2}, -v \right) (v^2+v)^{m-\frac{1}{2}} \,  \id v \times \nonumber \\
&&\!\!\!\!\!\!\! \2F1\left( m+ i t, m- i t, m+\frac{1}{2}, -u \right) \left(i t \right)_m \left(-i t \right)_m  \, \id t,
\end{eqnarray}
for $m \in \N_0$, where $(x)_n$ is the Pochhammer symbol, i.e. 
$$
(x)_n= \Gamma(x+n)/ \Gamma(x)=x(x+1)\ldots(x+n-1).
$$
\end{thm}
\subsection{The wave kernel near $x$ where the manifold is hyperbolic}
As we mentioned at the beginning of this section, in order to estimate the local counting function on the manifold $M$ we need to know the diagonal of the wave kernel on $M$. So far we obtained just the shifted wave kernel on $\Hy^n$. In this section we show how one may move from $\Hy^n$ to $M$ and change the shifted wave kernel to the wave kernel.

Let $(M^n,\metricg)$ be an $n$-dimensional Riemannian manifold, we say that $M^n$ is hyperbolic around $x \in M$ if there exists a positive radius $\rho$ such that the open metric ball, $B_\rho(x)$, centred at $x$ is isometric to the ball in hyperbolic $n$-space. Let us denote by $d(x)$  twice the maximal radius of that ball.

In order to distinguish the integral kernel of $  \int_\R \cos\left(\sqrt{ L_{M^n}} t \right) \funcg(t) \, \id t$ from the integral kernel of $  \int_\R \cos\left(\sqrt{ L_{\Hy^n}} t \right) \funcg(t) \, \id t$ we will be denote it by $k_{n,\funcg}(x,y)$. The unit wave propagation speed implies that the solution of the shifted wave equation (\ref{wave}) in $\Hy^n$ coincides with the solution in $B_\rho(x) \subset M^n$ for $t \in (-d(x),d(x))$, whenever initial function $f$ coincides on isometric balls and vanishes outside. In other words 
$
\tilde{k}_n(x,y) = k_n(x,y),
 $
whenever 
\begin{equation} \label{assumption}
\mathrm{supp}( \funcg) \subset (-d(x),d(x)).
\end{equation}
Therefore we can move from $\Hy^n$ to $M^n$ for sufficiently supported function. In addition, if $M^n$ is compact, $\kappa_j^2$ is the shifted eigenvalue of $\Delta_{M^n}$, satisfaing $\kappa_j^2= \lambda_j^2 -(n-1)^2/4$, $ \mathrm{arg}(\kappa_j) \in [-\pi/2, \pi/2)$, where $\Delta_{M^n} \varphi_j=\lambda_j^2 \varphi_j$. Then the integral kernel $k_n$ is just
\[ \sum_{j=0}^{\infty} h( \kappa_j) \varphi_j(x)\overline{ \varphi_j(y)},\]
which converges absolutely, uniformly in $x, y \in M^n$ for $\funcg \in C^{\infty}_0(\R)$. 

Let us assume from this point that  $\funcg \in C^{\infty}_0(-d(x),d(x))$. By Lemma~\ref{lem1}, the cosine  transform of the derivative of $N_{n,x}$ tested against a function $\funcg$  coincides with the diagonal of the wave kernel on $M^n$ tested against $\funcg$. Introducing functions $\tilde{h}, \tilde{\funcg}$ given by
 \begin{eqnarray}\label{tildeh}
\tilde{h}(t)= h(\sqrt{t^2+(n-1)^2/4}), \\
\tilde{\funcg}(t):= \frac{1}{2 \pi} \int_\R \tilde{h}(x) \cos(x t) \, \id x,  \label{tildeg}
\end{eqnarray}
 we have
\begin{equation} \label{24}
  \int_\R \int_\R \cos(\lambda t) N_{n,x}'(\lambda) \, \id \lambda\, \funcg(t) \, \id t = \int_\R \tilde{k}_n(t,x,x)\tilde{g}(t) \, \id t.
  \end{equation} 
It is surprising , if we look at formulas (\ref{tildeh}), (\ref{tildeg}), but it turns out that $\tilde{\funcg}$ is also smooth and compactly supported as well as $\funcg$, as it is proved in the technical lemma below. Therefore we can use Lemma~\ref{eventrace} applied to $\tilde{k}_{n,\tilde{\funcg}}$.\begin{lem} \label{lem}
Suppose that $\funcg$ is an even smooth and compactly supported function, such that $\mathrm{supp}(\funcg) \subset (-a,a)$ then $\tilde{\funcg}$ defined by (\ref{tildeh}), (\ref{tildeg}) is also smooth, compactly supported and even function such that $\mathrm{supp}(\tilde{\funcg}) \subset (-a,a)$.
\end{lem}
\begin{proof}
The Paley-Wiener theorem implies that the function $h$ is entire. For the sake of concreteness let us define this square root as  a square root on the Riemann surface and if $t \in i(0,(n-1)/2)$ then $\mathrm{arg}(t^2 + (n-1)^2/4) = 0$, if $t \in i(-(n-1)/2,0)$ then $\mathrm{arg}(t^2 + (n-1)^2/4) = -2 \pi$.

Let us define the auxiliary function $\gamma \colon \C \to \R$
\[ \gamma(t)= \frac{1+ (t^2+(n-1)^2/4)^{\frac{1}{2}}}{1+|t|}.\]
It turns out that this function is continuous and has a maximal value $(n+1)/2$ at 0, a minimal value at $t= i(n-1)/{2}$ and
\begin{equation} \label{minimum}
\inf_{t \in \C} \gamma(t) = \frac{2}{n+1}.
\end{equation}
Our choice of branch of the square root implies that the function $|\Im( t-(t+(n-1)^2 /4)^{1/2})|$ has  symmetries along the real and an imaginary axis. Moreover for $t=a+i b$, where $a,b \geq 0$ we have
\begin{eqnarray*}
| \Im( t-(t+(n-1)^2 /4)^{1/2}) | =&\\
&\!\!\!\!\!\!\!\!\!\!\!\!\!\!\!\!\!\!\!\!\!\!\!\!\!\!\!\!\!\!\!\!\!\!\!\!\!\!\!\!\!\!  b- \sqrt{ \frac{1}{2} \sqrt{ \left( a^2-b^2+\left( \frac{n-1}{2} \right)^2 \right)^2 + 4 a^2 b^2}-\frac{1}{2} \left( a^2-b^2+\left( \frac{n-1}{2} \right)^2 \right) }.
\end{eqnarray*}
This implies the following estimate 
\[
|\Im( t-(t+(n-1)^2 /4)^{1/2})| \leq \frac{n-1}{2},
\]
where the maximal value is reached for $t= \pm i(n-1)/2$. The Paley-Wiener theorem, triangle inequality and the estimate (\ref{minimum}) imply  
\begin{equation}
 |\tilde{h}(t)| \leq \frac{c_{\nu} e^{a|\Im( \sqrt{t^2+(n-1)^2/4})|}}{(1+\sqrt{t^2+(n-1)^2/4})^{\nu}} \leq \frac{c_{\nu} 2^{\nu} e^{\frac{a(n-1)}{2}} }{(n+1)^{\nu}} \frac{e^{a |\Im(t)|}}{(1+|t|)^{\nu}}.
 \end{equation}
 Therefore $\tilde{g} \in C_0^{\infty}(-a,a)$.
\end{proof}
This allows us to use already gathered information about the shifted wave kernel on $\Hy^n$ to control the local counting function on $M^n$. 
\begin{lem}\label{counting}
The cosine transform of the derivative of $N_{n,x}$ coincides in the interval $(-d(x),d(x))$ with the cosine transform of the derivative of $F_{n}$ for $n \in \N_+$, $n\geq 2$ where
\begin{eqnarray}
F'_{2}( \tau) &=& \frac{1}{2 \pi}H \! \left( \!\tau - \!\frac{1}{2} \right) \tanh \left(\!\pi \sqrt{\tau^2-\frac{1}{4}}\right) \! \tau \!,\nonumber \\
F_{3}'(\tau) & =& \frac{H(\tau-1)}{2 \pi^2}\tau \sqrt{\tau^2-1} ,\nonumber \\
F'_{2m+2}( \tau) &=& \frac{2 H \! \left( \tau -\! m-\!\frac{1}{2} \right)}{(4 \pi)^{m+1} m!} \tanh \left(\!\pi \sqrt{\tau^2-\left(m+\frac{1}{2}\right)^{\! 2}}\right) \! \tau \prod_{l=0}^{m-1} \left(\tau^2 -m^2+l^2-m+l\right),\nonumber \\
F_{2m+1}'(\tau) & =& \frac{2 H(\tau-m)}{(2 \pi)^{m+1}(2m-1)!!}\tau \sqrt{\tau^2-m^2} \prod_{l=1}^{m-1}\left(\tau^2-m^2+l^2\right),\nonumber \\
F_{n}(0) &=& 0 \nonumber.
\end{eqnarray}  
\end{lem}
\begin{proof}
By Lemma~\ref{lem} we have that the function $\tilde{\funcg}$ is smooth for a given smooth and compactly supported function $\funcg$, moreover its support is contained in the support of $\funcg$. therefore by (\ref{24}) and Lemma~\ref{eventrace} we have
$$
  \int_\R N_{n,x}'(\lambda) h(\lambda) \, \id \lambda =
  \begin{cases}
  \displaystyle{\frac{1}{(4 \pi)^{m+1}m!} \int_{\R} \tilde{h}(t)  \tanh( \pi t) t \left( \frac{1}{2} + i t \right)_{\!\!m} \!\! \left( \frac{1}{2} - i t \right)_{\!\!m} \!\! \id t,} & n=2m+2\\
\displaystyle{\frac{1}{(2 \pi)^{m+1}( 2m-1)!!}  \int_\R   \tilde{h}(t) \left(i t \right)_m\left(-i t \right)_m \, \id t,} & n=2m+1,
\end{cases}
$$
The statement is just a result of a coordinate change.
\end{proof}
\section{Estimates}
The goal of this section is to present explicit estimates of the local counting function on the manifold $M^n$. We will use explicit Fourier-Tauberian theorems due to Safarov which were used previously to estimate the local counting function for domains in $\R^n$. 

\begin{thm}[Safarov, \cite{Safarov}]
Let $F$ be a non-decreasing function on $\mathbb{R}$ with $\mathrm{supp}\;F \subset (0, + \infty)$. Suppose that the cosine transform of $F'(\tau)$ coincides on the interval $(-\delta,\delta)$ with the cosine transform of the function $n \tau_+^{n-1}$ then
\begin{eqnarray*}
F(\tau) &\geq &\tau^n - 2 \pi^{-1} \nu_{\lfloor \frac{n+2}{2} \rfloor}^2 n \delta^{-1} (\tau + \delta^{-1} \nu_{\lfloor \frac{n+2}{2} \rfloor})^{n-1},\\
F(\tau) &\leq& \tau^n +(2 \pi^{-1} \nu_{\lfloor \frac{n+2}{2} \rfloor}^2+\nu_{\lfloor \frac{n+2}{2} \rfloor}) n \delta^{-1} (\tau + \delta^{-1} \nu_{\lfloor \frac{n+2}{2} \rfloor})^{n-1},
\end{eqnarray*}
for all $\tau >0$, where $\nu_{m}$ denotes $2m$-th root of the first eigenvalue of $\Delta^m$ on the interval $(-1/2,1/2)$ subject to Dirichlet boundary conditions.
\end{thm}
Lemma~\ref{counting} shows that unfortunately the cosine transform of the derivative of the local counting function  on $M^n$ is not a polynomial of order $n-1$ but it behaves like a polynomial as $\lambda \to \infty$. The idea is to compare it to this asymptotic polynomial and estimate the error which we make by this comparison.
\begin{thm}\label{main}
Let $(M^n,\metricg)$ be a Riemannian manifold and let $x \in M$ be a point such that the metric ball of radius
$d(x)/2$ is isometric to a ball in $\mathbb{H}^n$. Then
the local counting function on $M^n$ satisfies the following estimates
\begin{eqnarray*}
N_{n,x}(\tau) & \leq & \frac{\omega_n}{(2 \pi)^n} \left[ \tau^{n} + \frac{n}{d(x)} \left( \frac{2}{\pi} \nu_{\lfloor \frac{n+2}{2}\rfloor}^2+\nu_{\lfloor \frac{n+2}{2}\rfloor}\right)\!\! \left( \tau+ \frac{\nu_{\lfloor \frac{n+2}{2}\rfloor}}{d(x)} \right)^{n-1} \right]
\end{eqnarray*}
For $n=2m+2$ we have
\begin{eqnarray*}
N_{2m+2,x}(\tau) &\geq & -\| G_{2m+2}\|_{\infty}+ \frac{\omega_{2m+2}}{(2 \pi)^{2m+2}}\left\{ \left( \tau -m -\frac{1}{2}\right)^{2m+2}- \left(m +\frac{1}{2} \right)^{2m+2} \nonumber \right. \\
& &\!\!\!\!\!\!\!\!\!\!\!\!\! \left.-\frac{(2m+2)\nu_{m+2}}{d(x)} \left[ \left( m+\frac{1}{2} + \frac{\nu_{m+2}}{d(x)} \right)^{2m+1} \!\!\!\!\!+ \frac{2 \nu_{m+2}}{\pi} \left( \tau + \frac{\nu_{m+2}}{d(x)} \right)^{2m+1}\right]\right\}\!,
\end{eqnarray*}
while for $n=2m+1$ we have
\begin{eqnarray*}
N_{2m+1,x}(\tau) &\!\!\!\!\! \geq & \!\!\!\!\!
\begin{cases}
\! \frac{\omega_{2m+1}}{(2 \pi)^{2m+1}} \!\left(\tau^{2m+1}- \frac{(4m+2) \nu_{m+1}^2}{d(x) \pi} \left( \tau + \frac{\nu_{m+1}}{d(x)} \right)^{2m} \right)\!- \|G_{2m+1}\|_{\infty}, \ \tau \in [0,m],\\
\! \frac{\omega_{2m+1}}{(2 \pi)^{2m+1}} \!\left((\tau -m)^{2m+1} - \frac{(4m+2) \nu_{m+1}^2}{d(x) \pi} \left( \tau + \frac{\nu_{m+1}}{d(x)} \right)^{2m} \right) \!-\|G_{2m+1}\|_{\infty}, \ \tau >m.
\end{cases}
\end{eqnarray*}
As before $\nu_{m}$ denotes $2m$-th root of the first eigenvalue of $\Delta^m$ on the interval $[-1/2,1/2]$ subject to Dirichlet boundary conditions, $\omega_n$ is the volume of $n$-dimensional Euclidean ball and constants $\|G_n\|_\infty$ are given by (\ref{negative}).
\end{thm}
The proof of this theorem is constructive and thus also is of interest itself as it gives an algorithm for computing estimates of the local counting function, estimates of the derivatives (in a space variable) of the local counting function as well. Some examples will be given in Example~\ref{ex}, we will also derive an estimate of the counting function on compact hyperbolic manifolds with an explicit remainder estimate (Corollary~\ref{counting1}).

Let us define a function $G_n$ as the error when we compare $F_n$ to a polynomial:
\[
\begin{cases}
G'_{2m+2}(\tau):= \frac{2 H \! \left( \tau -\! m-\!\frac{1}{2} \right)}{(4 \pi)^{m+1} m!} \left[ \tanh \left(\!\pi \sqrt{\tau^2-\left(m+\frac{1}{2}\right)^{\! 2}}\right) -1 \right] \! \tau\displaystyle{ \prod_{l=0}^{m-1}} \left(\tau^2 -m^2+l^2-m+l\right),
\\
G_{2m+2}(0)=0.
\end{cases}
\]
By the definition $G'_{2m+2}$ is clearly a integrable and negative function. Note that the function $\sqrt{t^2-m^2}$ has the following expansion
\[
\sqrt{t^2-m^2}= t- \frac{1}{2} \sum_{j=1}^{\infty} \frac{m^{2j} \left( \frac{1}{2} \right)_{j-1}}{t^{2j-1} j!},
\] where $(x)_n=\Gamma(x+n)/\Gamma(x)$ and the series converges for $|t|>m$. Let us define $G'_{2m+1}(\tau)$ as the part of the Laurent series of $F'_{2m+1}(\tau)$ with negative powers  of $\tau$ multiplied by the characteristic function of the set $[m,\infty)$.  Let us set $G_{2m+1}(0)=0$. The function $G_n$ satisfies the following bounds
\begin{equation} \label{negative}
\| G_n\|_{\infty} \leq \begin{cases}
\frac{1}{48 \pi}, &n=2, \\
\frac{1}{12 \pi^2}, & n=3,\\
\min \left\{\frac{2\left( \left(m-\frac{1}{2}\right)^2 +\frac{1}{4 \pi^2} \right)^m (2m+1)!}{\pi (4 \pi)^{m+2}m!}, \frac{2(2m+1)! e^{\pi (2m-1)}}{(16  \pi^3)^{m+1} m!} \right\}, & n=2m+2 \\
\frac{m^{2m+1}(1-m^4)+(2m-1)!m^2(1+m^{2m-2})}{(2 \pi)^{m+1}(2m-1)!!(1-m^4)}, &n=2m+1 \ \textrm{and } m \ \textrm{odd},\\
\frac{m^{2m+1}(1-m^4)+(2m-1)!m^4(1+m^{2m-4})}{(2 \pi)^{m+1}(2m-1)!!}, & n=2m+1 \ \textrm{and } m \ \textrm{even}.
\end{cases}.
\end{equation}
\begin{notation} \label{notation}
For the Dirichlet extension of $\Delta^m$ on the interval $[-1/2,1/2]$ let as before $\nu_m$ be the $2m$-th root of the first eigenvalue and 
$\zeta_m$ be the corresponding normalised eigenfunction extended by 0 to the real line. We define $\rho$ as the square of the Fourier transform of $\zeta_m$. Following \cite{Safarov} let
\begin{equation} \label{testrho} 
 \quad \rho_\delta( \tau)= \delta \rho (\delta \tau), \quad \rho_{ \delta, 0} = \delta \rho_{1,0} (\delta \tau), \quad \rho_{1,0}(\tau)= \int_{\tau}^{\infty} t \rho(t) \, \id t. 
\end{equation}
\end{notation}
Plancharel theorem implies that $\| \rho_{\delta} \|_1 = 1$.
The function $G'_{2m+1}$ is non-positive, therefore
\begin{equation} \label{infty}
 \rho_{\delta} \ast G_n( \tau) \leq 0, \quad \rho_{\delta,0} \ast G'_n\leq 0
\end{equation}
for $\delta>0$. By Young's inequality we have $ |\rho_{\delta} \ast G_n( \tau)| \leq \| G_n\|_{\infty}$. 
Note, that 
\begin{equation} \label{bounds}
\frac{n\omega_n}{(2 \pi)^n}H\left(\tau-\frac{n-1}{2}\right) \left(\tau- \frac{n-1}{2} \right)^{n-1} \leq F'_n(\tau)-G'_n(\tau)  \leq \frac{n\omega_n}{(2 \pi)^n}H(\tau) \tau^{n-1},
\end{equation}
where $\omega_n$ is a volume of $n$-dimensional Euclidean unit ball. 
In order to give bounds on the local counting function we will need the following Fourier Tauberian theorem.
\begin{lem}
\label{lemma}
Let $\rho$ be an even function such that $\mathrm{supp}(\hat{\rho})  \subset [-1,1]$ and suppose $|\rho(t)| \leq \mathrm{const.} (1+ t^2)^{-m-1}$ for $m> n/2$. Define $\rho_{\delta}$ as in (\ref{testrho}).  Suppose that $\mathrm{supp} E_n \in (0, + \infty)$ and the cosine transform of $E'_n(t)$ coincides on the interval $(-\delta, \delta)$ with the cosine transform of the function $ n H( t-\frac{n-1}{2}) (t-\frac{n-1}{2})^{n-1}$, then 
\begin{eqnarray}
\rho_{\delta} \ast \widetilde{E}_{2m+2}(\tau) &\!\!\!\!\!\!\!\!\!\!   \geq &\!\!\!\!\!\!\!\!\!\!  \left( \! \tau\! -\! m \!- \!\frac{1}{2} \! \right)^{2m+2}\!\!\!- \!\left( \! m \! +\! \frac{1}{2} \! \right)^{2m+2} \!\!\!- (2m \! +\! 2) \!\!\int \left| \frac{\nu}{\delta} \right| \! \left(\! m \!+\! \frac{1}{2} \!+\! \frac{\nu}{\delta} \! \right)^{2m+1} \!\!\!\!\rho (\nu) \id \nu , \nonumber \\
& & {\textrm{for }\tau  >0,} \nonumber  \\
\rho_{\delta} \ast \widetilde{E}_{2m+1} (\tau) &\geq &\begin{cases}
\tau^{2m+1} & \textrm{for } \tau \in [0,m],\\
(\tau-m)^{2m+1} & \textrm{for } \tau >m,
\end{cases} \nonumber
\end{eqnarray}
where $\tilde{E}(t)= E(t)-E(-t)$.
\end{lem}
\begin{proof}
The assumption implies that the Fourier transform of $\tilde{E_n}$ coincides on the interval $(-\delta, \delta)$ with the Fourier transform of $ H( t-\frac{n-1}{2}) ( t-\frac{n-1}{2})^{n}-H( -t-\frac{n-1}{2})( -t-\frac{n-1}{2})^{n}$, while the compactness of the support of the Fourier transform of $\rho$ implies that for $n=2m+2$ we have
\begin{eqnarray}
\rho_{\delta} \ast \widetilde{E}_{2m+2} (\tau) &= & \!\!\!\!\!\!\!\!\!\! \int \limits_{-\infty}^{\delta(\tau -m -\frac{1}{2})} \!\!\!\!\!\!\!\! \left(\tau - \frac{\nu}{\delta}-m-\frac{1}{2} \right)^{2m+2} \!\!\!\!\!\!\rho(\nu) \id \nu - \!\!\!\!\!\!\!\!\int \limits^{\infty}_{\delta(\tau +m +\frac{1}{2})} \!\!\!\!\!\!\!\!\left(\tau - \frac{\nu}{\delta}+m+\frac{1}{2} \right)^{2m+2} \!\!\!\!\!\!\!\rho(\nu) \id \nu \nonumber \\
&\!\!\!\!\!\!\!\!\!\! \!\!\!\!\!\!\!\!\!\! \!\!\!\!\!\!\!\!\!\! \!\!\!\!\!\!\!\!\!\! \!\!\!\!\!\!\!\!\!\! \!\!\!\!\!\!\!\!\!\!  \!\!\!\!\!\!\!\!\!\! \!\!\!\!\!\!\!\!\!\!\!\!\!\!= &\!\!\!\!\!\!\!\!\!\! \!\!\!\!\!\!\!\!\!\! \!\!\!\!\!\!\!\!\!\! \!\!\!\!\!\!\!\!\!\! \!\!\!\!\!\!\!\!\!\! \int\limits_{\R} \!\!\left(\! \tau \!- \! \frac{\nu}{\delta}\! - \! m \!-\! \frac{1}{2} \!\right)^{2m+2} \!\!\!\!\!\!\rho(\nu) \id \nu \!-\!\!\!\!\!\!\!\!\!\int\limits_{\delta(\!\tau \! - \! m \! - \! \frac{1}{2}\!)}^{\infty} \!\!\!\!\!\!\!\!\left(\! \tau \! -  \!\frac{\nu}{\delta} \!- \! m \!- \!\frac{1}{2} \! \right)^{2m+2} \!\!\!\!\!\!\rho(\nu) \id \nu -\!\!\!\!\!\!\!\!\int\limits^{\infty}_{\delta(\tau +m +\frac{1}{2})} \!\!\!\!\!\!\!\! \left(\!\tau \! -\! \frac{\nu}{\delta} \! + \! m \! + \! \frac{1}{2} \! \right)^{2m+2} \!\!\!\!\!\!\!\rho(\nu) \id \nu \nonumber \\
& \geq & \int\limits_{\R} \left(\tau - \frac{\nu}{\delta}-m-\frac{1}{2} \right)^{2m+2} \!\!\!\!\!\!\rho(\nu) \id \nu -\int\limits_{\R} \left( \frac{\nu}{\delta}-m-\frac{1}{2} \right)^{2m+2} \!\!\!\!\!\!\rho(\nu) \id \nu, \nonumber 
\end{eqnarray}
where the inequality is satisfied just for positive $\tau$. Let us define $P_n(\tau, \nu):= 1/2 [(\tau + \nu)^n+(\tau - \nu)^n]$. Then $P_n$ is a polynomial in $\tau$ and $\nu$ which contains only even powers of $\tau$. The fact that $\rho$ is an even function results in the inequality
\[ 
\rho_{\delta} \ast \widetilde{E} (\tau) \geq \int \limits_{\R} \left[ P_{2m+2}(\tau -m -1/2, \nu / \delta)-P_{2m+2}(m +1/2, \nu / \delta)\right] \rho(\nu) \, \id \nu
\]
The basic estimates 
\[ 
\tau^{2m+2} \leq P_{2m+2}(\tau,\nu) \leq \tau^{2m+2} +(2m+2)|\nu|(|\tau|+|\nu|)^{2m+1}
\]
finish the proof in the case of $n=2m+2$. For $n=2m+1$ we have
\begin{eqnarray}
\rho_{\delta} \ast \widetilde{E}_{2m+1} (\tau) &= &\int \limits_{-\infty}^{\delta(\tau -m )}  \left(\tau - \frac{\nu}{\delta}-m \right)^{2m+1} \!\!\!\!\!\!\rho(\nu) \id \nu + \int \limits^{\infty}_{\delta(\tau +m )} \left(\tau - \frac{\nu}{\delta}+m \right)^{2m+1} \!\!\!\!\!\!\!\rho(\nu) \id \nu \nonumber \\
&\!\!\!\!\!\!\!\!\!\! \!\!\!\!\!\!\!\!\!\! \!\!\!\!\!\!\!\!\!\! \!\!\!\!\!\!\!\!\!\! \!\!\!\!\!\!\!\!\!\!\!\!\!\!\!\!\!\!\!\!\!\!\!\!= &\!\!\!\!\!\!\!\!\!\! \!\!\!\!\!\!\!\!\!\! \!\!\!\!\!\!\!\!\!\! \!\!\!\!\!\!\!\!\!\! \int\limits_{\R} \!\!\left(\! \tau \!- \! \frac{\nu}{\delta}\! - \! m \right)^{2m+1} \!\!\!\!\!\!\rho(\nu) \id \nu \!+\!\!\!\!\!\int\limits_{\delta(\!\tau \! - \! m )}^{\infty} \!\!\left(\!\frac{\nu}{\delta} \! + \! m  \! -\tau \!  \right)^{2m+1} \!\!\!\!\!\!\rho(\nu) \id \nu +\!\!\!\!\!\!\!\!\int\limits^{\infty}_{\delta(\tau +m )} \!\!\!\! \left(\tau - \frac{\nu}{\delta}+m \right)^{2m+1} \!\!\!\!\!\!\!\rho(\nu) \id \nu \nonumber \\
& \geq & \int\limits_{\R} \left(\tau - \frac{\nu}{\delta}-m \right)^{2m+1} \!\!\!\!\!\!\rho(\nu) \id \nu +\int\limits_{\R} \left( m - \frac{\nu}{\delta}\right)^{2m+1} \!\!\!\!\!\!\rho(\nu) \id \nu, \nonumber 
\end{eqnarray}
for $\tau>0$. This implies the following estimates
\begin{equation}
\rho_{\delta} \ast \widetilde{E}_{2m+1} (\tau) \geq 
\begin{cases}
\tau^{2m+1} & \textrm{for } \tau \in [0,m],\\
(\tau-m)^{2m+1} & \textrm{for } \tau >m.
\end{cases}
\end{equation}
\end{proof}
At this point we have all tools needed to give the proof of Theorem~\ref{counting}:
\begin{proof}
One may prove that the test function defined in Notation~\ref{notation} satisfies the assumptions of the Theorem~1.3, Lemma~2.6 in \cite{Safarov} and Lemma~\ref{lemma}. Moreover Theorem~1.3 in \cite{Safarov}, our Theorem~\ref{counting} and compactness of the support of the function $\rho$ implies
\begin{equation}
\rho_{d(x)} \ast F_n(\tau) -\frac{d(x)^{-1}}{\int |\nu| \rho(\nu) \id \nu} \rho_{d(x),0} \ast F'_n(\tau) \leq N_{n,x} \leq \rho_{d(x)} \ast F_n(\tau) +\frac{d(x)^{-1}}{\int |\nu| \rho(\nu) \id \nu} \rho_{d(x),0} \ast F'_n(\tau)
\end{equation}
Let us define the asymptotic polynomial of $F_n$ as
\begin{equation}
p_a(n,\tau)= F_n(\tau) - G_n(\tau).
\end{equation}
Notice that the asymptotic polynomial is just the Taylor part of a Laurent series of $F_n$ multiplied by the characteristic function of the interval $\left[ \frac{n-1}{2}, \infty \right)$.  Safarov in \cite{Safarov} has shown that $\int |\nu| \rho(\nu) \id \nu \geq \pi/2 $. Inequalities (\ref{negative}) admit the estimate
\begin{equation} \label{upperbound}
N_{n,x} \leq \rho_{d(x)} \ast p_a(n,\tau) +\frac{2}{\pi d(x)} \rho_{d(x),0} \ast \partial_\tau p_a(n,\tau)
\end{equation}
A monotonicity of an integration and (\ref{bounds}) raise a conclusion
\begin{eqnarray}
N_{n,x}(\tau) & \leq & \frac{\omega_n}{(2 \pi)^n} \left[ \tau^{n} + \frac{n}{d(x)} \left( \frac{2}{\pi} \nu_{\lfloor \frac{n+2}{2}\rfloor}^2+\nu_{\lfloor \frac{n+2}{2}\rfloor}\right)\!\! \left( \tau+ \frac{\nu_{\lfloor \frac{n+2}{2}\rfloor}}{d(x)} \right)^{n-1} \right] \nonumber
\end{eqnarray}
Using one more time (\ref{negative}) one may show that
\begin{equation} \label{lowerbound}
\rho_{d(x)} \ast p_a(n,\tau) -\frac{2}{\pi d(x)} \rho_{d(x),0} \ast \partial_\tau p_a(n,\tau)+\rho_{d(x)} \ast G_n(\tau)  \leq N_{n,x}(\tau).
\end{equation}
When we take into account inequality (\ref{infty})  we arrive at
\begin{equation}
N_{n,x}(\tau) \geq \rho_{d(x)} \ast p_a(n,\tau) -\frac{2}{\pi d(x)} \rho_{d(x),0} \ast \partial_\tau p_a(n,\tau) -\|G_n\|_{\infty}.
\end{equation}
Inequality~(\ref{bounds}) and Lemma~\ref{lemma} complete the proof.
\end{proof}
In the special cases $n=2,3,4$ slightly stronger estimates can be obtained.
\begin{exm}\label{ex}
Recall that $p_a(2,\tau)= H\left(\tau - \frac{1}{2}\right)\frac{\tau}{2 \pi}  $, $p_a(3,\tau)= H(\tau-1)\left(\frac{\tau^3}{6 \pi^2} -\frac{\tau}{4 \pi^2}\right)$, $p_a(4,\tau):=  H(\tau-\frac{3}{2}) \left( \frac{\tau^4}{32 \pi^2}- \frac{\tau^2}{8 \pi^2} \right)$. Inequality (\ref{upperbound}) with Corollary~2.3 from \cite{Safarov} imply
\begin{eqnarray}
N_{2,x}(\tau)& \leq & \frac{1}{4 \pi} \left(\tau^2 + \frac{4\nu^2_2 +2 \nu_2 \pi}{\pi d(x)} \left(\tau + \frac{\nu_2}{d(x)} \right) \right),  \label{2gorne}\\
N_{3,x}(\tau) & \leq & \frac{1}{6 \pi^2} \left( \tau^3 + \left( \frac{6 \nu_2^2+3 \pi \nu_3}{\pi d(x)} \right) \left( \tau+ \frac{\nu_2}{d(x)} \right)^2 \right)- \frac{1}{4 \pi^2}  \left( \tau - \frac{2 \nu_1^2}{\pi d(x)}\right),\\
 N_{4,x}(\tau) & \leq & \!\!\frac{1}{32 \pi^2} \!\left( \!\tau^4 \!+ \!\left( \frac{8 \nu_3^2\!+\!4 \pi \nu_3}{\pi d(x)} \right)\! \left( \!\tau \! + \! \frac{\nu_3}{d(x)} \!\right)^3 \right)\! - \! \frac{1}{8 \pi^2} \! \left( \! \tau^2 \! - \! \frac{4 \nu_2^2}{\pi d(x)} \! \left( \!\tau \! + \! \frac{\nu_2}{d(x)} \! \right) \! \right)
\end{eqnarray}
Moreover, one may compute the exact value for the supremum norm of $G_n$, 
\[ \| G_2\|_{\infty} =  \frac{1}{48 \pi}, \quad \| G_3\|_{\infty} =  \frac{1}{12 \pi^2}, \quad \| G_4\|_{\infty} =  \frac{17}{7680 \pi^2}\]
\end{exm}
These estimates combined with inequality (\ref{lowerbound})  and Corollary~2.3 from \cite{Safarov} give
\begin{eqnarray}
N_{2,x}(\tau)& \geq & \frac{1}{4 \pi} \left(\tau^2 - \frac{4\nu^2_2 }{\pi d(x)} \left(\tau + \frac{\nu_2}{d(x)} \right) -\frac{1}{12} \right), \label{2dolne} \\
N_{3,x}(\tau) & \geq & \frac{1}{6 \pi^2} \left( \tau^3 - \frac{6 \nu_2^2}{\pi d(x)}  \left( \tau+ \frac{\nu_2}{d(x)} \right)^2 \right)- \frac{1}{4 \pi^2}  \left( \tau + \frac{2 \nu_1^2+\pi \nu_1}{\pi d(x)}\right)- \frac{1}{12 \pi^2},\\
N_{4,x}(\tau) & \geq & \frac{1}{32 \pi^2} \left( \tau^4 - \frac{8 \nu_3^2}{\pi d(x)} \left( \tau+ \frac{\nu_3}{d(x)} \right)^3 \right)- \nonumber \\
& & \frac{1}{8 \pi^2}  \left( \tau^2 + \frac{4 \nu_2^2+2 \pi \nu_2}{\pi d(x)}  \left( \tau+ \frac{\nu_2}{d(x)} \right) \right)- \frac{17}{7680 \pi^2},
\end{eqnarray}
where the numerical values of $\nu_n$ are
\begin{eqnarray*}
\nu_2 =  4.73004074\ldots, \\
\nu_3 =  6.28318530\ldots, \\
\nu_4 = 7.81870734 \ldots .
\end{eqnarray*}
Notice that simple integration of the presented estimates gives bounds on the counting function of the Laplacian for hyperbolic manifolds. Denote by $l_n$ the length of the shortest closed geodesic on $M^n$. The following theorem gives a proof of the Weyl's law, moreover it gives the estimate on the remainder term. 
\begin{cor}\label{counting1}
The counting function  of Laplacian on a compact hyperbolic manifold $M^n$ satisfies the following estimates
\begin{eqnarray}
N_{n}(\tau) & \leq & \frac{\omega_n|M^n|}{(2 \pi)^n} \left[ \tau^{n} + \frac{n}{l_n} \left( \frac{2}{\pi} \nu_{\lfloor \frac{n+2}{2}\rfloor}^2+\nu_{\lfloor \frac{n+2}{2}\rfloor}\right)\!\! \left( \tau+ \frac{\nu_{\lfloor \frac{n+2}{2}\rfloor}}{l_n} \right)^{n-1} \right]
\end{eqnarray}
For $n=2m+2$ we have
\begin{eqnarray}
\frac{N_{2m+2}(\tau)}{|M^{2m+2}|} &\geq & -\| G_{2m+2}\|_{\infty}+ \frac{\omega_{2m+2}}{(2 \pi)^{2m+2}}\left\{ \left( \tau -m -\frac{1}{2}\right)^{2m+2}- \left( m +\frac{1}{2} \right)^{2m+2} \nonumber \right. \\
& &\!\!\!\!\!\!\!\!\!\!\!\!\!\!\!\!\!\!\!\!\!\!\! \!\!\!\!\!\!\!\!\!\!  \left.-\frac{(2m+2)\nu_{m+2}}{l_{2m+2}} \left[ \left( m+\frac{1}{2} + \frac{\nu_{m+2}}{l_{2m+2}} \right)^{2m+1} + \frac{2 \nu_{m+2}}{\pi} \left( \tau + \frac{\nu_{m+2}}{l_{2m+2}} \right)^{2m+1}\right]\right\},
\end{eqnarray}
while for $n=2m+1$ we have
\begin{eqnarray}
\frac{N_{2m+1}(\tau)}{|M^{2m+1}|} &\!\!\!\!\!\geq & \!\!\!\!\!
\begin{cases}
\!\frac{\omega_{2m+1}}{(2 \pi)^{2m+1}} \! \left(\! \tau^{2m+1} \! - \! \frac{(4m+2) \nu_{m+1}^2}{l_{2m+1} \pi} \left( \tau + \frac{\nu_{m+1}}{l_{2m+1}} \right)^{2m} \right)- \|G_{2m+1}\|_{\infty}, \tau \in [0,m],\\
\!\frac{\omega_{2m+1}}{(2 \pi)^{2m+1}} \!\left( \! (\tau \! - \! m)^{2m+1} \! - \! \frac{(4m+2) \nu_{m+1}^2}{l_{2m+1} \pi} \! \left(\! \tau \! + \! \frac{\nu_{m+1}}{l_{2m+1}} \! \right)^{2m} \right) \! - \! \|G_{2m+1}\|_{\infty}, \tau >m.
\end{cases}
\end{eqnarray}
\end{cor} 
\subsection{Estimates of eigenfunctions and heat trace for the Laplacian on compact manifolds hyperbolic near $x$ }
The above estimates can be used to obtain information about the eigenfunctions.
Let us assume that $\lambda^2$ is an eigenvalue of the Laplace operator on a closed manifold $M^n$ that is hyperbolic near $x$. Suppose 
$\varphi$ is a corresponding normalized eigenfunction, then by the definition of a local counting function we have,
\begin{equation} \label{pomm}
|\varphi(x) |^2 \leq  \limsup_{\tau \to \lambda^+}  N_{n,x}(\tau)-  \liminf_{\tau \to \lambda^-}  N_{n,x}(\tau)
\end{equation}
where the equality holds only for single eigenvalues. Let us use inequalities (\ref{upperbound}), (\ref{lowerbound}) one more time to show
\begin{equation}
|\varphi(x) |^2 \leq \frac{4}{\pi d(x)} \rho_{d(x),0} \ast \partial_\lambda p_a(n,\lambda)+\|G_n\|_{\infty}.
\end{equation} 
Let us summarise these considerations.
\begin{cor}
Let $(M^n,\metricg)$ be a Riemannian manifold and let $x \in M$ be a point such that the metric ball of radius
$d(x)/2$ is isometric to a ball in $\mathbb{H}^n$.
Let $\varphi$ be an eigenfunction of the Laplacian on  $M^n$  with eigenvalue $\lambda^2$, then
\begin{equation}
|\varphi(x)|^2 \leq \frac{8 n\ \nu_{\lfloor\frac{n+2}{2} \rfloor}^2 \omega_n}{ d(x)(2 \pi)^{n+1}} \left(\lambda+ \frac{\nu_{\lfloor\frac{n+2}{2} \rfloor}}{d(x)} \right)^{n-1}+\|G_n\|_{\infty},
\end{equation}
where the supremum norm of $G_n$ is given by (\ref{negative}).
\end{cor}

Similarly bounds for the local heat trace may be obtained. For $t> 0$ the local heat trace $k_t(x)$
is defined as the diagonal of the integral kernel of the heat operator and can be expressed in terms of the eigenfunctions as
 \[
 k_t(x)= \sum_{\lambda_j \geq 0} e^{-\lambda_j^2 t} |\varphi_j(x)|^2.
 \] Following \cite{Strohmaier} let us denote by $R_t^c$ the remainder of a truncated  series at $c>0$:
 \begin{equation}
 R_t^c(x)= \sum_{\lambda_j^2 > c} e^{-\lambda_j^2 t} |\varphi_j(x)|^2.
 \end{equation} 
It is often necessary to estimate this quantity in numerical computations if only finitely many eigenvalues are available.
The quantity $R_t^c(x)$ represents the error made when the expansion series for the local heat trace is truncated.
Let us introduce a rescaled counting function $\tilde{N}_x(\tau):= N_x(\sqrt{\tau})$. Then, the remainder, $R_t^c$, in terms of rescaled counting function is given by $ R_t^c(x)= \int_{c^+}^{\infty} \tilde{N}_x'(\tau) e^{-t \tau} \, \id \tau.$ By integration by parts we obtain 
 \begin{equation} \label{heattrace1}
 R_t^c(x)=- \lim_{\varepsilon \to 0^+} \tilde{N}_x(c+ \varepsilon) e^{-ct} + t \int\limits_c^{\infty} \tilde{N}_x(\tau) e^{-t \tau} \, \id \tau.
 \end{equation}
As usual the incomplete gamma function, $\Gamma: \C \times \R_+ \to \C $, is defined by
\[
\Gamma(z,r):= \int\limits_r^{\infty} e^{-t}t^{z-1} \, \id t.
\]
 Of course $\Gamma(z,0)$ is just the usual gamma function. Equation (\ref{heattrace}) together with our estimates imply the following bound on the remainder
\begin{equation}
 R_t^c(x)\! \leq \! - \! \lim_{\varepsilon \to 0^+} \tilde{N}_x(c+ \varepsilon) e^{-ct}+ c_1 t^{-\frac{n}{2}} \Gamma \left( \frac{n}{2} \! + \! 1,tc \right)+ c_1 c_2 \sum_{l=0}^{n-1} \binom{n-1}{l} c_3^{n-1-l} t^{-\frac{l}{2}} \Gamma \left(\! \frac{l}{2} \! +\! 1, tc \! \right)\!,
\end{equation}
where $c_1= \omega_n/(2 \pi)^n$, $ c_2= n(2 \nu^2_{\lfloor \frac{n+2}{2}\rfloor}+ \pi\nu_{\lfloor \frac{n+2}{2}\rfloor} )/ (d(x) \pi) $, $c_3= \nu_{\lfloor \frac{n+2}{2}\rfloor}/d(x) $.  Since
$$
k_t(x)= \lim_{c \to 0} R^c_t(x) + \sum_{\lambda_j=0} | \varphi_j(x)|^2,
$$
the estimate (\ref{pomm}) implies the following theorem:
\begin{thm}\label{heattrace}
Let $(M^n,\metricg)$ be a Riemannian manifold and let $x \in M$ be a point such that the metric ball of radius
$d(x)/2$ is isometric to a ball in $\mathbb{H}^n$. 
 Then the local heat trace satisfies the estimate
\begin{eqnarray*}
k_t(x) &\!\!\!\!\! \leq & \!\!\!\!\!\! \frac{\omega_{n}}{(2 \pi)^{n}} \! \left[ \! \Gamma \! \left( \! \frac{n \! + \! 2}{2} \! \right)\! t^{-\frac{n}{2}} \! + \!  \frac{n \Gamma \! \left( \! \frac{n+1}{2} \! \right) \! (2 \nu^2_{\lfloor \frac{n+2}{2}\rfloor} \!+ \!\pi \nu_{\lfloor \frac{n+2}{2}\rfloor})} {d(x) \pi } \! \left( \!\frac{1}{ \sqrt{t}} \! + \! \frac{\nu_{\lfloor \frac{n+2}{2}\rfloor}}{d(x)} \! \right)^{n-1} \! \right] \! ,
\end{eqnarray*}
for $x$ in the ball locally isometric to the hyperbolic $n$-space, $d(x)$ is a twice of the maximal radius of this ball.
\end{thm}
For hyperbolic manifolds integration over $x$ gives a bound on the heat trace. 
 \begin{cor}
Let $M^n$ be a closed connected hyperbolic manifold of dimension $n$. Then the heat trace  $k_t = \mathrm{tr}(e^{-t \Delta})$ satisfies the estimate
\begin{eqnarray*}
k_t &\!\!\!\!\! \leq & \!\!\!\!\!\! \frac{\omega_{n}|M^n|}{(2 \pi)^{n}} \! \left[ \! \Gamma \! \left( \! \frac{n \! + \! 2}{2} \! \right)\! t^{-\frac{n}{2}} \! + \!  \frac{n \Gamma \! \left( \! \frac{n+1}{2} \! \right) \! (2 \nu^2_{\lfloor \frac{n+2}{2}\rfloor} \!+ \!\pi \nu_{\lfloor \frac{n+2}{2}\rfloor})} {l_n \pi } \! \left( \!\frac{1}{ \sqrt{t}} \! + \! \frac{\nu_{\lfloor \frac{n+2}{2}\rfloor}}{l_n} \! \right)^{n-1} \! \right] \!,
\end{eqnarray*}
where $l_n$ is the length of the shortest geodesic in $M^n$.
 \end{cor}
 \subsection{Estimates of the derivatives of eigenfunctions}
In the previous subsection obtained bounds for the eigenfunctions of the Laplace operator. A similar technique may be used to obtain bounds
for the derivatives of the eigenfunctions.  Assuming that $M^n$ is hyperbolic near $x$ notice that 
\begin{equation}\label{pochodna}
\sum_{j=0}^{\infty} h(r_j) |\vec{v} \varphi_j(x)|^2=- \frac{1}{2} {\partial_u} \left. \tilde{k}_n(x,y) \right|_{u=0} ,
\end{equation}
where $\vec{v}$ is a tangent vector at a point $x$.
\begin{dfn}
Let $\nabla^l \varphi$ denotes $l$-th covariant derivative of $ \varphi \in C^{\infty}(M^n)$. Define the absolute value of $l$-th covariant derivative  of $\varphi$ by
\begin{equation}
 |\nabla^l \varphi|^2 := \metricg^{i_1 j_1} \metricg^{i_2 j_2} \ldots \metricg^{i_l j_l} (\nabla^l \varphi)_{i_1 i_2 \ldots i_l} \overline{(\nabla^l \varphi)_{j_1 j_2 \ldots j_l} }\end{equation}
for $l \in \N^+$. 
\end{dfn}
Define
\[N_{n,x}^l(\tau):=\sum_{j=0}^{\infty} H(\tau^2 - \lambda_j^2)| \nabla^l \varphi_j(x)|^2.\]
Equation (\ref{pochodna}) implies the following theorem about the function $N_x^1$.
\begin{lem}\label{counting2}
The cosine transform of the derivative of $N^1_{n,x}$ coincides in the interval $(-d(x),d(x))$ with the cosine transform of the derivative of $F^1_{n}$ for $n \in \N_+$, $n\geq 2$ where
\begin{eqnarray}
{F^1_{2m+2}}'( \tau) &\!\!\!\!\!\!=& \!\!\!\!\!\! \frac{2 H \! \left( \tau -\! m-\!\frac{1}{2} \right)}{(4 \pi)^{m+1} m!} \tanh \left(\!\pi \sqrt{\tau^2-\left(m+\frac{1}{2}\right)^{\! 2}}\right) \! \tau^3 \prod_{l=0}^{m-1} \left(\tau^2-m^2+l^2-m+l\right) \!,\nonumber \\
{F^1_{2m+1}}'(\tau) & \!\!\!\!\!\!=&\!\!\!\!\!\! \frac{2 H(\tau-m)}{(2 \pi)^{m+1}(2m-1)!!}\tau^3 \sqrt{\tau^2-m^2} \prod_{l=1}^{m-1}(\tau^2+l^2-m^2),\nonumber \\
F_{n}(0) &=& 0.
\end{eqnarray}  
\end{lem}
This Lemma gives us a tool to estimate the first derivative of an eigenfunction; an application of this result may be found in~\cite{Strohmaier}. Let us define ${G^{1}_{n}} '(\tau)$ as the part of the Laurent series of ${F^1_{n}}'(\tau)$ with negative powers  of $\tau$ multiplied by a characteristic function of the set $[(n-1)/2,\infty)$. Set $G^1_{2m+1}(0)=0$. The function ${G^1_{n}}'$ is non-positive, therefore
\begin{equation} \label{infty2}
 \rho_{\delta} \ast {G^1_n}( \tau) \leq 0, \quad \rho_{\delta,0} \ast {G^1_n}'\leq 0
\end{equation}
for $\delta>0$. By Young's inequality we obtain
\begin{equation} \label{negative2}
 |\rho_{\delta} \ast G_n^1( \tau)| \leq \| G_n^1\|_{\infty}. 
 \end{equation}
Moreover, one can show that $G_n^1$ is bounded and
\begin{equation} 
 \| G_n^1\|_{\infty} \!\! \leq \!\! \begin{cases}
\frac{17}{1920 \pi}, &\!\!\!\!  n=2, \\
\frac{11}{240 \pi^2}, & \!\!\!\! n=3, \\
\min \left\{\frac{2\left( \left(m-\frac{1}{2}\right)^2 +\frac{1}{4 \pi^2} \right)^m (2m+3)!}{\pi (4 \pi)^{m+4}m!},  \frac{e^{2 \pi(m+1/2)} (2m+3)!}{m! 2^{4m+5}} \pi^{2m+5} \right\}\!, & \!\!\!\!  n=2m+2, \\
\frac{11}{60}\frac{m^{2m+3}(1-m^4)+(2m-1)!m^4(1+m^{2m-2})}{(2 \pi)^{m+1}(2m-1)!!(1-m^4)}, &\!\!\!\!\!\!\!\!\!\!\!\! n=2m+1 \ \textrm{and } m \ \textrm{odd},\\
\frac{11}{60}\frac{m^{2m+3}(1-m^4)+(2m-1)!m^6(1+m^{2m-4})}{(2 \pi)^{m+1}(2m-1)!!(1-m^4)}, &\!\!\!\!\!\!\!\!\!\!\!\! n=2m+1 \ \textrm{and } m \ \textrm{even}.
\end{cases}
\end{equation}
By the definition of $N^1_{n,x}$ we have
\begin{equation}
|\nabla \varphi(x) |^2 \leq \overline{ \lim_{\tau \to \lambda^+} } N_{n,x}^1(\tau)- \underline{ \lim_{\tau \to \lambda^-} } N_{n,x}^1(\tau),
\end{equation}
where equality holds only for simple eigenvalues. By a similar argument as before
\begin{equation}
|\nabla \varphi(x)|^2 \leq \frac{8 (n+2)\ \nu_{\lfloor\frac{n+4}{2} \rfloor}^2 \omega_n}{ d(x)(2 \pi)^{n+1}} \left(\lambda+ \frac{\nu_{\lfloor\frac{n+4}{2} \rfloor}}{d(x)} \right)^{n+1}+ \| G_n^1 \|_{\infty}.
\end{equation}
\begin{exm}
Suppose that $\varphi$ is an eigenfunction of the Laplacian wit eigenvalue $\lambda^2$ on a 
compact Riemannian manifold $M^n$ hyperbolic near $x \in M^n$
and let $d(x)$ be as before. Then
\begin{eqnarray*}
|\nabla \varphi(x)|^2 &\leq & \frac{4\ \nu_{3}^2}{ d(x)\pi^2} \left(\lambda+ \frac{\nu_{3}}{d(x)} \right)^{3}+ \frac{17}{1920 \pi} \quad \textrm{ for } n=2 \\
|\nabla \varphi(x)|^2 & \leq & \frac{10 \ \nu_{3}^2}{ 3 d(x)\pi^3} \left(\lambda+ \frac{\nu_{3}}{d(x)} \right)^{4}+ \frac{11}{240 \pi^2} \quad \textrm{ for } n=3 \\
|\nabla \varphi(x)|^2 & \leq & \frac{3 \nu_{4}^2 }{ 4 d(x)\pi^3} \left(\lambda+ \frac{\nu_{4}}{d(x)} \right)^{5}+ \frac{367}{64512 \pi^2} \quad \textrm{ for } n=4 
\end{eqnarray*}
\end{exm}
\subsection{Estimates of higher derivatives of eigenfunctions on hyperbolic surfaces} 
 Adopting the method from~\cite{Strohmaier}, we will derive bounds on the higher derivatives of the eigenfunctions on hyperbolic surfaces. By the Selberg pre-trace formula and finite propagation speed, the cosine transform of $\partial_\tau N^2_{2,x}( \tau)$ coincides with the cosine transform of the function $\partial_\tau F_2^2(\tau)$ on the interval $(-d(z),d(z))$, where
\[F_{2}^2(0)=0, \quad {F_{2}^2}'(\tau)= \frac{1}{ 2 \pi} H( \tau^2 -1/4)( |\tau|^3 + |\tau|^5) \tanh(\pi \sqrt{\tau^2 -1/4}).\]
Taking into account the facts that 
\begin{eqnarray*}
\int_{1/2}^{\infty} \tau^3( \tanh(\pi \sqrt{\tau^2- 1/4}) -1)\, \id \tau &=& -\frac{17}{960}, \\
\int_{1/2}^{\infty} \tau^5( \tanh(\pi \sqrt{\tau^2- 1/4}) -1) \, \id \tau&=& -\frac{407}{40320}, 
\end{eqnarray*}
we get the following estimate:
\begin{equation} \label{eqn.2} -\frac{29}{1260 \pi} \leq \sign(\tau) G_{2}^2( \tau) \leq 0, \end{equation}
where $G_{2}^2( \tau):= F_{2}^2(\tau)- \sign( \tau) \frac{1}{8 \pi} \tau^4 - \sign(\tau) \frac{1}{12 \pi} \tau^6$. The Fourier Tauberian Theorem 1.3 in~\cite{Safarov} and estimate~(\ref{eqn.2}) implies the following estimates for $N^2_{2,x}$.
\begin{thm}
The function $N^2_{2,x}$ satisfies 
\begin{eqnarray*} 
N_{2,x}^2( \tau)& \leq & \frac{1}{12 \pi} \left( \tau^6 +  \frac{12 \nu_4^2 + 6 \pi \nu_4}{\pi d(x)} \left( \tau + \frac{\nu_4}{d(x)} \right)^5 \right) \\
& &+ \frac{1}{8 \pi} \left( \tau^4 + \frac{8 \nu_3^2 +4 \pi \nu_3}{d(x) \pi }\left( \tau + \frac{\nu_3}{d(x)} \right)^3 \right)+ \frac{29}{1260 \pi},\\
N_{2,x}^2( \tau)& \geq & \frac{1}{12 \pi} \!\left(\! \tau^6 \! -\! \frac{12 \nu_4^2 }{\pi d(x)} \! \left( \! \tau \! +\!  \frac{\nu_4}{d(x)} \! \right)^5 \right) + \frac{1}{8 \pi} \left( \tau^4 -\frac{8 \nu_3^2 }{d(x) \pi }\left( \tau + \frac{\nu_3}{d(x)} \right)^3 \right)- \frac{29}{1260 \pi}.
\end{eqnarray*}
\end{thm}
The theorem above implies the following estimate on the derivatives of eigenfunctions.
\begin{cor}
Let $\varphi$ be a normalised eigenfunction of the Laplace operator on a compact manifold $M^2$ with eigenvalue $\lambda$. Then
\begin{eqnarray*} 
 |\nabla^2 \varphi(x) |^2  & \leq & \frac{1}{12 \pi} \left(  \frac{24 \nu_4^2 + 6 \pi \nu_4}{\pi d(z)} \left( \lambda + \frac{\nu_4}{d(x)} \right)^5 \right) \\
& \ & + \frac{1}{8 \pi} \left(  \frac{16 \nu_3^2 +4 \pi \nu_3}{d(x) \pi }\left( \lambda + \frac{\nu_3}{d(x)} \right)^3 \right)+ \frac{29}{630 \pi}.
\end{eqnarray*}
\end{cor} 

The same method can be used to get bounds for higher derivatives of eigenfunctions. Again by Selberg's pre-trace formula and finite propagation speed one gets that  the cosine transform of $\partial_{\tau} N^3_{2,x}( \tau)$ coincides with the cosine transform of the function $\partial_{\tau} F_{2}^3(\tau)$ on the interval $(-d(x),d(x))$, where
\[ F_{2,3} (0)=0, \quad F_{2,3} '(\tau)= \frac{1}{2 \pi} H( \tau^2 -1/4)(4 |\tau|^3 +3 |\tau|^5+ |\tau|^7) \tanh(\pi \sqrt{\tau^2 -1/4}).\]
Note that 
\[ \int_{1/2}^{\infty} \tau^7( \tanh(\pi \sqrt{\tau^2- 1/4}) -1) \, \id \tau = -\frac{1943}{215040}, \]
and so
\begin{equation} \label{eqn.3}
 \| G_{2}^3 \|_{\infty} =\frac{2467}{26880 \pi},
 \end{equation}
  where $G_{2}^3(\tau)= F_{2}^3(\tau) - \sign(\tau) \frac{1}{16 \pi}\tau^8 -\sign(\tau)\frac{1}{4 \pi}  \tau^6 - \sign(\tau) \frac{1}{2 \pi}\tau^4 $. Taking into account  Fourier Tauberian Theorem 1.3 in~\cite{Safarov} and estimate~(\ref{eqn.3}) we obtain the following.
  \begin{thm}
The function $N^3_{2,x}$ satisfies 
\begin{eqnarray*} 
N_{2,x}^3( \tau)& \leq & \frac{1}{16 \pi} \! \left( \! \tau^8 \! + \! \frac{16 \nu_5^2 + 8 \pi \nu_5}{\pi d(x)} \! \left( \! \tau \! + \! \frac{\nu_5}{d(x)} \! \right)^7 \right) \! + \! \frac{1}{4 \pi} \! \left( \! \tau^6 \! + \!  \frac{12 \nu_4^2 + 6 \pi \nu_4}{\pi d(x)} \! \left( \! \tau \! + \! \frac{\nu_4}{d(x)} \! \right)^5 \right) \\
& \ &+ \frac{1}{2 \pi} \left( \tau^4 + \frac{8 \nu_3^2 +4 \pi \nu_3}{d(x) \pi }\left( \tau + \frac{\nu_3}{d(x)} \right)^3 \right)+ \frac{2467}{26880 \pi},\\
N_{2,x}^3( \tau)& \geq &  \frac{1}{16 \pi} \left( \tau^8 -  \frac{16 \nu_5^2}{\pi d(x)} \left( \tau + \frac{\nu_5}{d(x)} \right)^7 \right)  +\frac{1}{4 \pi} \left( \tau^6 -  \frac{12 \nu_4^2 }{\pi d(x)} \left( \tau + \frac{\nu_4}{d(x)} \right)^5 \right) \\
& \ &+ \frac{1}{2 \pi} \left( \tau^4 - \frac{8 \nu_3^2 }{d(x) \pi }\left( \tau + \frac{\nu_3}{d(x)} \right)^3 \right)- \frac{2467}{26880 \pi}.
\end{eqnarray*}
\end{thm}

\begin{cor}
Let $\varphi$ be a normalised eigenfunction of the Laplace operator with eigenvalue $\lambda^2$. Then
\begin{eqnarray*} 
 |\nabla^3 \varphi(x) |^2  & \leq & \frac{1}{ 2 \pi^2 d(x)} \left( (4 \nu_5^2 +  \pi \nu_5)\left( \lambda + \frac{\nu_5}{d(x)} \right)^7 
+ (12 \nu_4^2 + 3 \pi \nu_4) \left( \lambda + \frac{\nu_4}{d(x)} \right)^5  +\right. \\
& \ & \left.  \left(16 \nu_3^2 +4 \pi \nu_3\right)\left( \lambda + \frac{\nu_3}{d(x)} \right)^3\right)+ \frac{2467}{13440  \pi}.
\end{eqnarray*}
\end{cor} 
Similarly the cosine transform of $\partial_{\tau} N^l_z(\tau)$ coincides with the cosine transform of the function $\partial_{\tau} F_2^l(\tau)$ on the interval $(-d(z),d(z))$, where
\begin{eqnarray*} 
F_{2}^l(0) & \!\!\!\!\!=  &\!\!\!\!\! 0,\quad \textrm{ for } l=4,5,6,7,8 \\
\partial_\tau {F_{2}^4}(\tau)& \!\!\!\!\!=  &\!\!\!\!\!  \frac{1}{2 \pi} H( \tau^2 -1/4) \tanh(\pi \sqrt{\tau^2 -1/4})(32 |\tau|^3 +23 |\tau|^5+6 |\tau|^7+ |\tau|^9),\\
\partial_\tau {F_{2}^5}(\tau)& \!\!\!\!\!=  &\!\!\!\!\!  \frac{1}{2 \pi} H( \tau^2 -1/4) \tanh(\pi \sqrt{\tau^2 -1/4})(328 |\tau|^3 \! + \! 280 |\tau|^5 \! + \! 75 |\tau|^7 \! + \! 10|\tau|^9 \!+\!  |\tau|^{11}),\\
\partial_\tau {F_{2}^6}(\tau) & \!\!\!\!\!=  &\!\!\!\!\!   \frac{1}{2 \pi} H( \tau^2 -1/4)\tanh(\pi \sqrt{\tau^2 -1/4})\left(5752 |\tau|^3 \! +\! 5040 |\tau|^5 \! +\! 1399 |\tau|^7\! + \!185 |\tau|^9 \! +\! \right. \\
& & \left. 15 |\tau|^{11}+ |\tau|^{13}\right) ,\\
\partial_\tau {F_{2}^7}(\tau) & \!\!\!\!\!=  &\!\!\!\!\!  \frac{1}{2 \pi} H( \tau^2 -1/4)\tanh(\pi \sqrt{\tau^2 -1/4}) \left( 140944 |\tau|^3 +125864 |\tau|^5+36096 |\tau|^7+ \right. \\
& & \left. 4893 |\tau|^9+ 385 |\tau|^{11}+ 21 |\tau|^{13}+  |\tau|^{15}\right) \\
\partial_\tau {F_{2}^8}(\tau) & \!\!\!\!\!=  &\!\!\!\!\!  \frac{1}{2 \pi} H( \tau^2 -1/4)\! \tanh(\pi \sqrt{\tau^2 -1/4}) \! \left(488 3472|\tau|^3\! + \! 4419704 |\tau|^5\! +\! 1299288 |\tau|^7 \! +\right. \\
& & \left. 181275 |\tau|^9+ 7231 |\tau|^{11}+ 189 |\tau|^{13}+  15 |\tau|^{15} + |\tau|^{17}\right).
\end{eqnarray*}
This information might be used to give  bounds in up to the $C^8$ norm for the eigenfunction on a hyperbolic manifold of dimension 2.
\section{Applications}
As an example of an application, we show how the contribution to the length spectrum in Selberg's trace formula can be bounded. For a compact hyperbolic surface $M$ Selberg's trace formula states (see e.g. \cite{Marklof})
\begin{equation} \label{htrace}
\mathrm{tr}(e^{-\Delta t})=\frac{|M| e^{-t/4}}{4 \pi t} \int_0^{\infty} \frac{\pi e^{-r^2t}}{\cosh^2(\pi r)}\id r + \frac{e^{-t/4}}{2\sqrt{ \pi t}} \sum_{n=1}^{\infty} \sum_{\gamma}  \frac{l(\gamma) e^{\frac{-n^2 l(\gamma)^2}{4t}}}{2 \sinh\frac{n l(\gamma)}{2}},
\end{equation}
where the sum is over the set of primitive closed geodesics $\gamma$.
The first term can be computed and does not depend on the geometry of the manifold. Let us denote by $l$ the length of the shortest closed geodesic. Then because each term in (\ref{htrace}) is positive, the second term is bounded for $t<T<\sqrt{l^2+1}-1$ by 
\[
F_T(t)=\sqrt{\frac{T}{t}} \mathrm{tr}(e^{-\Delta T}) e^{\frac{T}{4}+\frac{l^2}{4 T}} e^{\frac{-l^2}{4 t}},
\]
and rapidly decreasing in $t$ as $t \to 0^+$. Our  estimates~(\ref{2dolne}), (\ref{2gorne}) of the local counting function and the equality~(\ref{heattrace1}) imply that we have the following bound on a function $R_t^c$:
\begin{equation} \label{Rct}
R_t^c(x) \leq \frac{1}{4 \pi} \left[ \frac{\Gamma(2,tc)}{t} +\frac{c_1 \Gamma(3/2, tc)}{ t^{1/2}}+ {e^{-ct}}\left( -c + \frac{c^{1/2}4 \nu_2^2 }{\pi d(x)} +\frac{8\nu_2^3+2 \nu_2^2 \pi}{\pi d(x)^2} + \frac{1}{12}\right) \right],
\end{equation}
where $\Gamma$ is the incomplete gamma function, $c_1=\frac{(4 \nu_2^2+2\nu_2 \pi)}{ \pi d(x)}$, $d(x)$ is a twice of the injectivity radius and $\nu_2$ is the first nonzero solution to the equation $\cos(\lambda) \cosh(\lambda)=1$. 
\begin{exm}
The Bolza surface is a compact hyperbolic surface of genus 2 which maximises the order of the symmetry group in this genus. The shortest  simple closed geodesic has length $2 \cosh(1 +\sqrt{2})$, whereas the volume of the manifold is $4 \pi$. On can see that the integral over the manifold of $R_t^c(x)$ gives us the error which we make computing the heat trace when we know only finitely many eigenvalues. The estimate for $R^c_t$ blows up near the origin but it gives good approximation for large $t$, which may been seen on the plot of $R_t^c$ when only 20 first eigenvalues are known. We used the list of eigenvalues obtained by the second author and Uski~\cite{Strohmaier}, which may be found at: \url{http://homepages.lboro.ac.uk/~maas3/publications/eigdata/eig-bolza-refined0-1000.dat}.
\begin{figure}
\centering
\includegraphics[width=7cm]{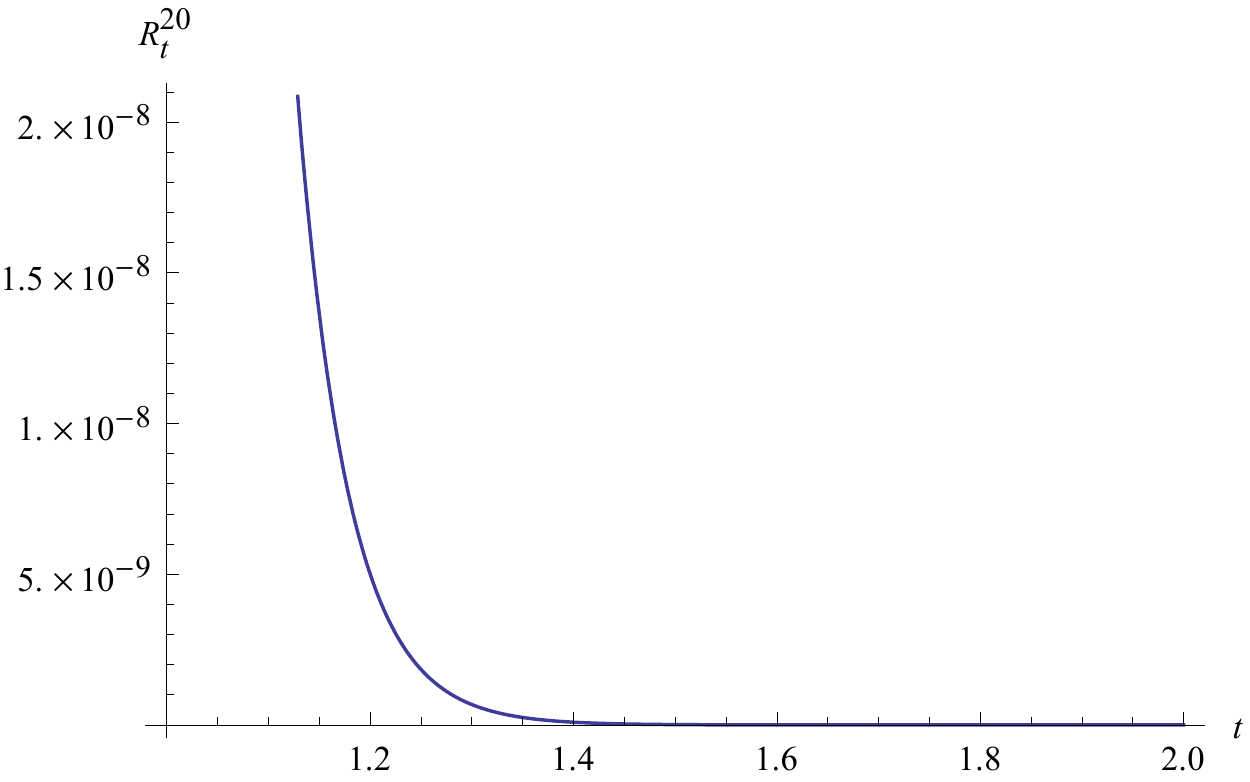}
\quad
\includegraphics[width=7cm]{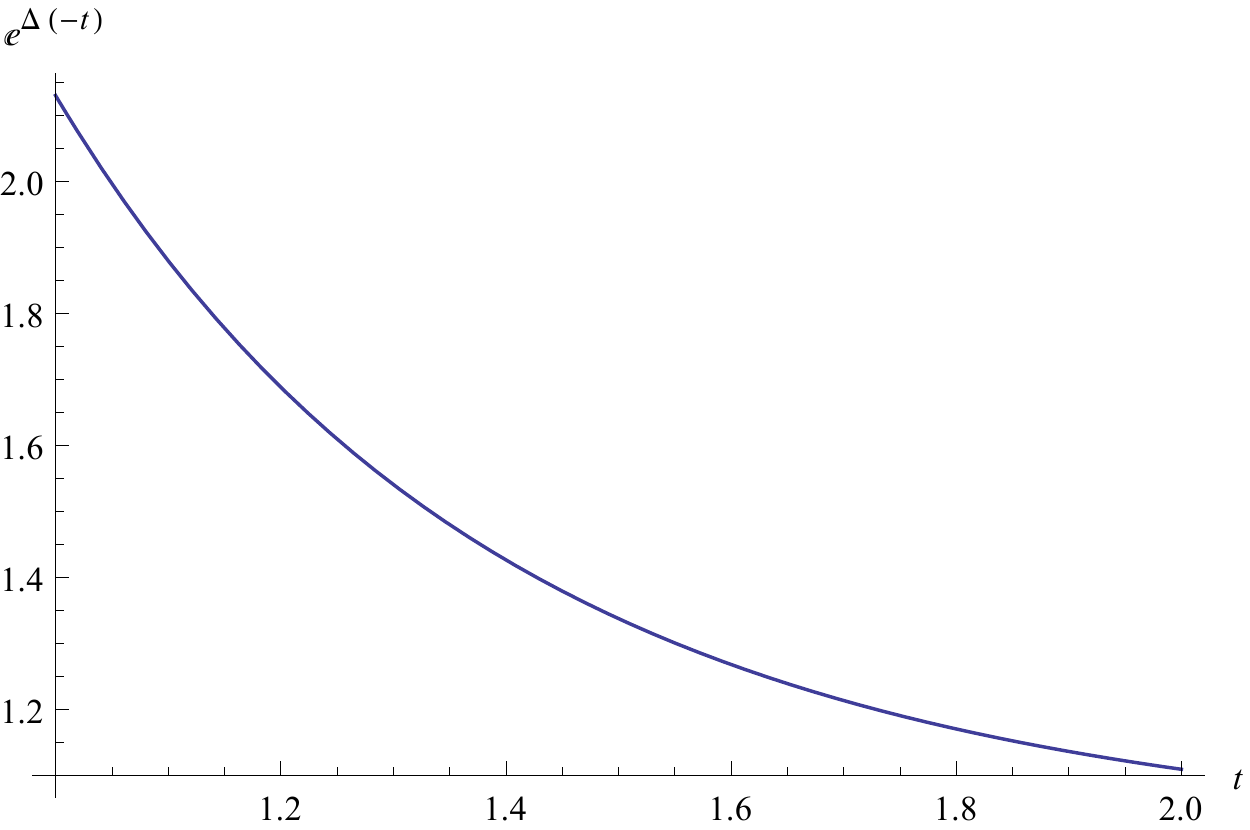}
\caption{The function $R_t^{20}$ on the left and the approximation of the heat trace on the right.}
\label{fig:obrazek k}
\end{figure}
\end{exm} 
Because we obtained the estimate of the local heat trace, we have for a hyperbolic surface of genus $\mathtt{g}$
$$
F_T(t) \! \leq \!\frac{\mathtt{g}-1}{\sqrt{t}} e^{\frac{T}{4}+\frac{l^2}{4 T}-\frac{l^2}{4 t}} \! \left( \frac{1}{ \sqrt{T}} +\frac{2 \nu_2^2+\nu_2 \pi}{ \sqrt{\pi} l } +\sqrt{T}\left(\frac{4\nu_2^3+2 \nu_2^2 \pi}{ \pi l^2} \right) \right).
$$
This shows that for small $t>0$ the main contribution to the heat trace is from the first term of the expansion (\ref{htrace}), which does not depend on the geometry of the manifold. By our heat trace estimates we can actually estimate this quantity.

Estimates may also be obtained for the spectral determinant of $\Delta$. Let us consider the spectral zeta function, $\zeta_{\Delta}(s)$, defined as the meromorphic continuation of the function
$$
\zeta_{\Delta}(s)= \sum_{\lambda_i \neq 0}^{\infty} \lambda_i^{-2s},
$$
where $\lambda_i^2$ are the eigenvalues of the Laplace operator. Then the zeta-regularized determinant $\mathrm{det}_\zeta(\Delta)$ of the Laplacian is defined by
$$
\log (\mathrm{det}_{\zeta} (\Delta))= - \zeta_{\Delta}'(0).
$$
Because 0 is not a pole of $\zeta_\Delta$ the zeta-regularized determinant is well defined. The second author  and Uski~\cite{Strohmaier} showed that
$$
\zeta_{\Delta}'(0)= L_1^\epsilon+L_2^\epsilon+L_3^\epsilon,
$$
where
\begin{eqnarray*}
\!\!\!\!\!\!\!\!\!\!\!\!\!\!\!\!\!\!\!\! L_1^\epsilon &=& \sum_{i=1}^{\infty}\Gamma(0,\epsilon \lambda_i^2),\\
\!\!\!\!\!\!\!\!\!\!\!\!\!\!\!\!\!\!\!\! L_2^{\epsilon}&=&- \frac{|M|}{4 \pi \epsilon} - \left(\frac{|M|}{12 \pi} +1\right)(\gamma + \log(\epsilon)) +  \\
&&  \!\!\!\!\!\!\!\!\!\!\!\!\!\!\!\!\!\!\!\!\!\!\! \frac{|M|}{4}\int_0^{\infty} \mathrm{sech}^2(\pi r) \left( \frac{1-\mathrm{E}_2(\epsilon(r^2+1/4))}{\epsilon}+ \left( r^2 + \frac{1}{4} \right) (\gamma-1+\log(\epsilon(r^2+1/4)))\right) \id r,\\
&& \!\!\!\!\!\!\!\!\!\!\!\!\!\!\!\!\!\!\!\! L_3^\epsilon= \sum_{n=1}^{\infty} \sum_{\gamma} \int_0^\epsilon e^{-\frac{t}{4}} \frac{l(\gamma) e^-\frac{n^2 l(\gamma)^2}{4 t}}{4 \sqrt{\pi} t^{3/2} \sinh(\frac{1}{2}n l(\gamma))} \id t.
\end{eqnarray*}
Here $\mathrm{E}_2(x)$ is the generalised exponential integral which is given by the formula
$$
\mathrm{E}_n(x)= \int_1^{\infty} e^{-xt} t^{-n} \, \id t,
$$
where $\gamma$ is the Euler constant. Let us assume that we have the list of eigenvalues up to $c$, i.e., we know $\{ \lambda_i^2| \lambda_i^2\leq c\}$  and we know the constant $l$. Then we can give a good estimate for $\log \mathrm{det}_\zeta(\Delta)$. That integral $L_2^\epsilon$ can be evaluated with high accuracy with numerical integration. We can split $L_1$ into
$$
L_1^{\epsilon}= \sum_{0<\lambda_i^2 \leq c} \Gamma(0, \epsilon \lambda_i^2) + \int_{\epsilon}^{\infty} t^{-1}R^c_t \, \id t. 
$$
By the (\ref{Rct}) we know that the integral in the formula above is bounded by
\begin{eqnarray*}
(g-1) \left[ \left( -c + \frac{\sqrt{c}4 \nu_2^2 }{\pi l} +\frac{8\nu_2^3+2 \nu_2^2 \pi}{\pi l^2} + \frac{1}{12}\right)\Gamma(0,c \epsilon) + \frac{e^{-c \epsilon}}{\epsilon} + \frac{\Gamma(\frac{1}{2}, c \epsilon)(4 \nu_2^2 +2 \nu_2 \pi)}{\sqrt{\epsilon} \pi l}\right].
\end{eqnarray*}
Note that for $\epsilon\leq T$ we have
$$
|L_3^\epsilon| \leq \int_0^\epsilon \frac{F_T(t)}{2 t} \id t,
$$
therefore $L^\epsilon_3$ is bounded by
$$
\frac{\mathtt{g}-1}{l} e^{\frac{T}{4}+\frac{l^2}{4 T}} \! \left( \frac{1}{ \sqrt{T}} +\frac{2 \nu_2^2+\nu_2 \pi}{ \sqrt{\pi} l } +\sqrt{T}\left(\frac{4\nu_2^3+2 \nu_2^2 \pi}{ \pi l^2} \right) \right) \Gamma \left( \frac{1}{2},\frac{l^2}{4 \epsilon}\right).
$$
Summarizing we have found:
\begin{cor}
The spectral determinant of the Laplace operator on a compact, connected hyperbolic manifold of dimension 2 and genus $\mathtt{g}$ satisfies the lower estimate:
\begin{eqnarray*}
&&\!\!\!\!\! \!\!\!\!\! \!\!\!\!\! - \log(\mathrm{det}_\zeta(\Delta)) \leq \\
&&\!\!\!\!\! \!\!\!\!\! \!\!\!\!\!   \pi(\mathtt{g}-1)  \int_0^{\infty} \mathrm{sech}^2(\pi r) \!\! \left( \! \frac{1 \! - \! \mathrm{E}_2(\epsilon(r^2 \! + \! \frac{1}{4}))}{\epsilon} \! + \! \left( \! r^2 \! + \! \frac{1}{4} \right) \! \left(\! \gamma \! - \! 1 \! + \! \log\!\left( \!\epsilon\!\left( \!r^2 \! + \! \frac{1}{4}\!\right)\!\right)\!\right) \!\! \right)\!\! \id r \!-\\
&&\!\!\!\!\! \!\!\!\!\! \!\!\!\!\!  \frac{(\mathtt{g}-1)}{ \epsilon}  -  \left(\frac{g+2}{3}\right)(\gamma  +  \log(\epsilon))  +  \sum_{0<\lambda_i^2 \leq c} \Gamma(0, \epsilon \lambda_i^2)+ \\
&&\!\!\!\!\! \!\!\!\!\! \!\!\!\!\! (\mathtt{g}-1) \left[ \left( -c + \frac{\sqrt{c}4 \nu_2^2 }{\pi l} +\frac{e^{-c\epsilon}}{ \epsilon} +\frac{8\nu_2^3+2 \nu_2^2 \pi}{\pi l^2} + \frac{1}{12}\right)\Gamma(0,c \epsilon)  + \frac{\Gamma(\frac{1}{2}, c \epsilon)(4 \nu_2^2 +2 \nu_2 \pi)}{\sqrt{\epsilon} \pi l}\right] \! +\\
&&\!\!\!\!\! \!\!\!\!\! \!\!\!\!\! \frac{\mathtt{g}-1}{l} e^{\frac{T}{4}+\frac{l^2}{4 T}} \! \left( \frac{1}{ \sqrt{T}} +\frac{2 \nu_2^2+\nu_2 \pi}{ \sqrt{\pi} l } +\sqrt{T}\left(\frac{4\nu_2^3+2 \nu_2^2 \pi}{ \pi l^2}  \right) \right) \Gamma \left( \frac{1}{2},\frac{l^2}{4 \epsilon}\right).
\end{eqnarray*}
An upper bound is given by:
\begin{eqnarray*}
&&\!\!\!\!\! \!\!\!\!\! \!\!\!\!\! - \log(\mathrm{det}_\zeta(\Delta)) \geq \\
&&\!\!\!\!\! \!\!\!\!\! \!\!\!\!\!   \pi(\mathtt{g}-1)  \int_0^{\infty} \mathrm{sech}^2(\pi r) \!\! \left( \! \frac{1 \! - \! \mathrm{E}_2(\epsilon(r^2 \! + \! \frac{1}{4}))}{\epsilon} \! + \! \left( \! r^2 \! + \! \frac{1}{4} \right) \! \left(\! \gamma \! - \! 1 \! + \! \log\!\left( \!\epsilon\!\left( \!r^2 \! + \! \frac{1}{4}\!\right)\!\right)\!\right) \!\! \right)\!\! \id r \!-\\
&&\!\!\!\!\! \!\!\!\!\! \!\!\!\!\!  \frac{(\mathtt{g}-1)}{ \epsilon}  -  \left(\frac{g+2}{3}\right)(\gamma  +  \log(\epsilon))  +  \sum_{0<\lambda_i^2 \leq c} \Gamma(0, \epsilon \lambda_i^2)- \\
&&\!\!\!\!\! \!\!\!\!\! \!\!\!\!\! \frac{\mathtt{g}-1}{l} e^{\frac{T}{4}+\frac{l^2}{4 T}} \! \left( \frac{1}{ \sqrt{T}} +\frac{2 \nu_2^2+\nu_2 \pi}{ \sqrt{\pi} l } +\sqrt{T}\left(\frac{4\nu_2^3+2 \nu_2^2 \pi}{ \pi l^2}  \right) \right) \Gamma \left( \frac{1}{2},\frac{l^2}{4 \epsilon}\right),
\end{eqnarray*}
for $0<\epsilon \leq T < \sqrt{l^2+1}-1$ and $c>0$.
\end{cor}
\begin{exm}
In the example of the Bolza surface the  inverse exponent of the sum of the first term in $L^\epsilon_1$ and $L^\epsilon_2$ yields the numerical value
$4.73115$.
With the same $c=20$ as before and $\epsilon =0.3524 $, $T= 2.2165$ one obtains upper and the lower bounds $4.88303$ and $4.51591$ respectively. For $c=50$, $\epsilon=0.22161$, $T= 2.2165$  upper and lower estimates are  $4.71927$ and $4.7253$ respectively. The known value is
$$
\mathrm{det}_\zeta(\Delta) \approx 4.72273280444557.
$$
\end{exm}


\begin{thebibliography}{9}
\bibitem{Hormander2} L. H\"{o}rmander: \emph{The spectral function of an elliptic operator}, Acta Math. 121, 1968,193-218,
\bibitem{Duistermaat} J. Duistermaat, V. Guillemin: \emph{The spectrum of positive elliptic operators and periodic bicharacteristics.} Inventiones math. 29, 39-79 (1975),
\bibitem{Sogge} C. Sogge: \emph{Concerning the $L^p$ norm of spectral for second-order elliptic operators on compact manifolds}. J. Funct. Anal., 77(1):123-138, 1988,
\bibitem{Tacy} A. Hassell, M. Tacy: \emph{Improvement of eigenfunction estimates on manifolds of nonpositive curvature}, arXiv:1212.2540v1, 2012,
\bibitem{Iwaniec} H. Iwaniec: \emph{Spectral Methods of Automorphic Forms}, American Mathematical Society, 1997,
\bibitem{Grigor} A. Grigor'yan and M. Noguchi: \emph{The heat kernel on hyperbolic space}, Bull. London Math. Soc. 30 (6) Pp: 643-650, 1998,
\bibitem{Safarov} Y. Safarov: \emph{Fourier Tauberian theorems and applications},  Journal of Functional Analysis 185, 2001,
\bibitem{Marklof} J. Marklof: \emph{Selberg's trace formula: An Introduction}, Hyperbolic Geometry and Applications in Quantum Chaos and Cosmology, eds. J. Bolte and F. Steiner, Cambridge University Press 2011, pp. 83-119,
\bibitem{Strohmaier} A. Strohmaier, V. Uski: \emph{An algorithm for the computation of eigenvalues, spectral zeta functions and zeta-determinants on hyperbolic surfaces},  Comm. Math. Phys. 317 (2013), no. 3, 827Ð869,
\bibitem{Helgason} S. Helgason: \emph{Geometric Analysis on Symmetric Spaces}, American Mathematical Society, 2008,
\bibitem{Berard} P. H. B\'erard: \emph{On the Wave Equation on a Compact Riemannian Manifold without Conjugate Points}, Mathematische Zeitschrift, Springer-Verlag, 1977,
\bibitem{Chavel} I. Chavel: \emph{Eigenvalues in Riemannian Geometry}, Academic Press, 1984.
\end{thebibliography}
\end{document}